\documentclass[12pt]{scrartcl}
 \usepackage{cmap}
 \usepackage{amsmath,amsxtra,amscd,amssymb,latexsym,stmaryrd,amsthm, mathtools, mathrsfs, dsfont}
  \usepackage{xcolor,enumitem}

\usepackage{lmodern} 
\usepackage[utf8]{inputenc}
\usepackage[T1]{fontenc}

\usepackage[colorlinks=true,linkcolor=red!70!black,citecolor=green!70!black, urlcolor=magenta!70!black,backref]{hyperref}

\addtokomafont{title}{\rmfamily\itshape}

\theoremstyle{plain}
\newtheorem{theo}{Theorem}[section]
\newtheorem*{theo*}{Theorem}
\newtheorem{coro}[theo]{Corollary}
\newtheorem{prop}[theo]{Proposition}
\newtheorem{lemm}[theo]{Lemma}
\newtheorem{theomain}{Theorem}
\newtheorem{coromain}[theomain]{Corollary}

\theoremstyle{definition}
\newtheorem{defi}[theo]{Definition}

\newcommand*{\dd}%
  {\relax\ifnum\lastnodetype>0\mskip\medmuskip\fi\mathrm{d}}

\newcommand{\Id}{\operatorname{Id}}

\newcommand{\Hol}{\operatorname{Hol}}
\newcommand{\Ck}{\operatorname{\mathscr{C}}}

\newcommand{\mdim}{\operatorname{mdim}_M}
\newcommand{\mdKS}{\operatorname{mdim}_{KS}}

\newcommand{\udimM}{\operatorname{dim}^+_M}

\newcommand{\ptt}[1]{\mathcal{P}_{#1}}
\newcommand{\prob}{\operatorname{\mathscr{P}}}
\newcommand{\card}{\operatorname{Card}}

\newcommand{\one}{\mathds{1}}

\def\R{\mathbb{R}}

\def\Om{\Omega}
\def\al{\alpha}

\def\se{\subseteq}

\newcommand\lr[1]{\left( #1 \right)}

\newcommand\lrs[1]{\left\{ #1 \right\}}

\title{An invitation to rough dynamics: zipper maps}

\author{Beno\^{\i}t R. Kloeckner \textsuperscript{$*$} \and Nicolae Mihalache \thanks{Univ Paris Est Creteil, Univ Gustave Eiffel, CNRS, LAMA UMR8050, F-94010 Creteil, France} 
}

\begin{document}

\maketitle

\begin{abstract}
In the field of dynamical systems, it is not rare to meet irregular functions, which are typically Hölder but not Lipschitz (e.g. the Weierstrass functions). Our goal is to scratch the surface of the following question: what happens if we consider irregular maps and \emph{iterate} them?

We introduce the family of \emph{zipper maps}, which are irregular in the above sense, and study some of their dynamical properties. For a large set of parameters, the corresponding zipper map admits horseshoe of all orders; as an immediate consequence, every order on $k\ell$ points can be realized by $k$ orbits of length $\ell$ of the map.

These maps have infinite topological entropy, and we refine this statement by showing that they have positive metric mean dimension with respect to the Euclidean metric, as well as by introducing other notions of higher complexity.

Finally, we prove that every interval map (thus including zipper maps) have vanishing absolute metric mean dimension, proving a small case of the conjecture that the absolute metric mean dimension coincides with the topological mean dimension.
\end{abstract}

\setcounter{tocdepth}{1}
\tableofcontents

\section{Introduction}

Dynamical systems form by now an incredibly broad and deep field, with many different kinds of systems having been studied  thoroughly. Even restricting to continuous maps acting on a compact space, one could spend a lifetime learning about them. However, it seems that there is still a blind spot in the literature: maps of low regularity, by which we mean less than locally lipschitz, typically not more than Hölder-continuous. The goal of this article is to showcase  the kind of dynamical properties that one can observe in one specific family of such irregular maps.

\subsection{Zipper maps}

Let us first introduce the family of maps we shall be interested in. Their graph will be a \emph{zipper curve} (hence the name);
as functions, they where introduced by Bruneau \cite{Bruneau1974variation} as extremal points in certain functional spaces;  until now they seem not to have been considered as maps that one can iterate. They have the advantage of being relatively explicit while exhibiting quite wild dynamical properties.

By $\Ck^\alpha([0,1])$ we mean the Banach space of $\alpha$-H\"older functions
defined on $[0,1]$ with values in $\mathbb{R}$ endowed with its usual norm
\[\lVert f\rVert_\alpha = \sup_{x\neq y} \frac{\lvert f(x) - f(y)\rvert}{\lvert x-y\rvert^\alpha} + \sup_x \lvert f(x)\rvert\]
and by $\Ck^\alpha_0$ the convex, closed subset of the functions $f$ with range $[0,1]$ and such that $f(0)=0$ and $f(1)=1$ (which shall then be seen as point-valued maps rather than scalar-valued functions). When $\alpha=0$ we mean the space of continuous functions or maps, endowed with the supremum norm.

Let $p=\big((x_1,y_1),(x_2,y_2)\big)\in(0,1)^2$ be a pair of points in the unit
square such that $x_2>x_1$ and $y_2<y_1$, and let $\Phi_p:\Ck^0_0 \to \Ck^0_0$ be the map defined by
\[\Phi_p f(x) = \begin{dcases*} y_1 f\big(\frac{x}{x_1}\big) & if $x\in[0,x_1]$ \\               
      y_1-(y_1-y_2) f\big(\frac{x-x_1}{x_2-x_1}\big) & if $x\in[x_1,x_2]$ \\
          y_2+(1-y_2) f\big(\frac{x-x_2}{1-x_2}\big) & if $x\in[x_2,1]$
              \end{dcases*} \] 
Then $\Phi_p$ is a contraction in the uniform norm, of ratio \[v_{\max}=\max(y_1,y_1-y_2,1-y_2)<1,\]
and thus has a unique fixed point $Z_p\in \Ck^0_0$, which we shall call the \emph{zipper map} of parameter $p$; the case $p=\big((.3,.7),(.8,.1) \big)$ is shown in Figure \ref{f:zigzag}.

\begin{figure}[htp]
\centering
\includegraphics[width=.5\linewidth]{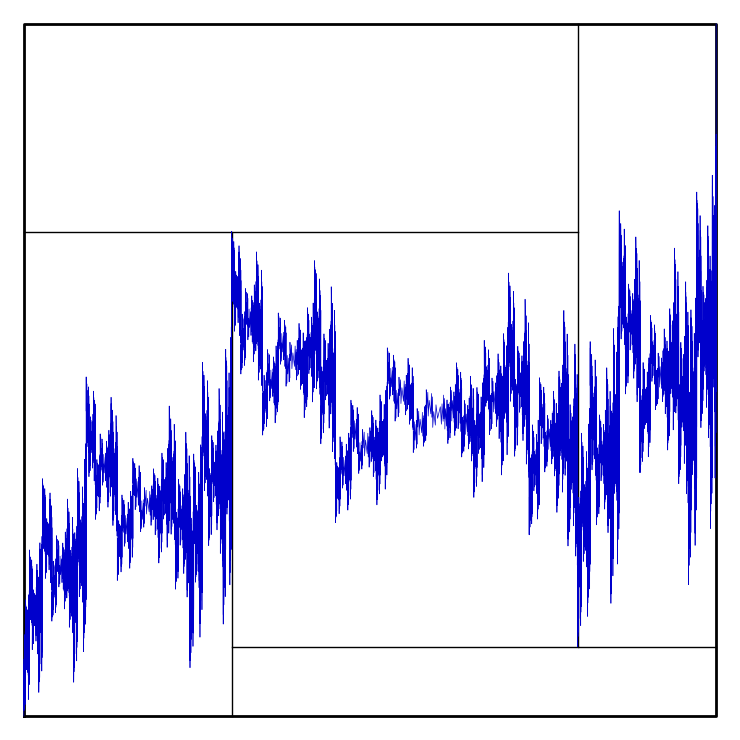}
\caption{The graph of a zipper map. Note that $1$ is mapped to $1$, which is barely visible given the extremely high speed of variation at that point.}\label{f:zigzag}
\end{figure}

We will further restrict to the \emph{hypersensitive} case where the (piecewise affine) image of the identity map by $\Phi_p$ has all its slopes greater than $1$ in absolute value:
\begin{equation}1 < \lambda_{\min} := \min\Big(\frac{y_1}{x_1}, \frac{y_1-y_2}{x_2-x_1}, \frac{1-y_2}{1-x_2}\Big)
\label{eq:lambdamin}
\end{equation}

\subsection{Main results}

Topological entropy is one of the paradigmatic measurement of chaos, and shifts on finite alphabets are among the most basic models for chaotic maps; as we will recall below, for continuous interval maps they are strongly related through ``horseshoes''. Our first result shows that many zipper maps exhibit a strong form of chaos by having horseshoes of arbitrarily order.

\begin{theomain}\label{thm:horseshoes}
Let $T=Z_p$ be an hypersensitive zipper map and additionally assume either one of these conditions:
\begin{enumerate}
\item\label{enumi:horse1} $p$ is symmetric with respect to the center of the square, i.e. $x_1+x_2=y_1+y_2=1$,
\end{enumerate}
or
\begin{enumerate}[resume]
\item\label{enumi:horse2} the pair of second coordinates $(y_1,y_2)$ of $p$ lies in the open set
\[B = \{(y_1,y_2)\in(0,1)^2 \mid y_1^2>y_2 \text{ and } y_1>(2-y_2)y_2 \}.\]
\end{enumerate}
Then $T$ admits horseshoes of all orders.
\end{theomain}

\begin{figure}[htp]
\centering
\includegraphics[width=.5\linewidth]{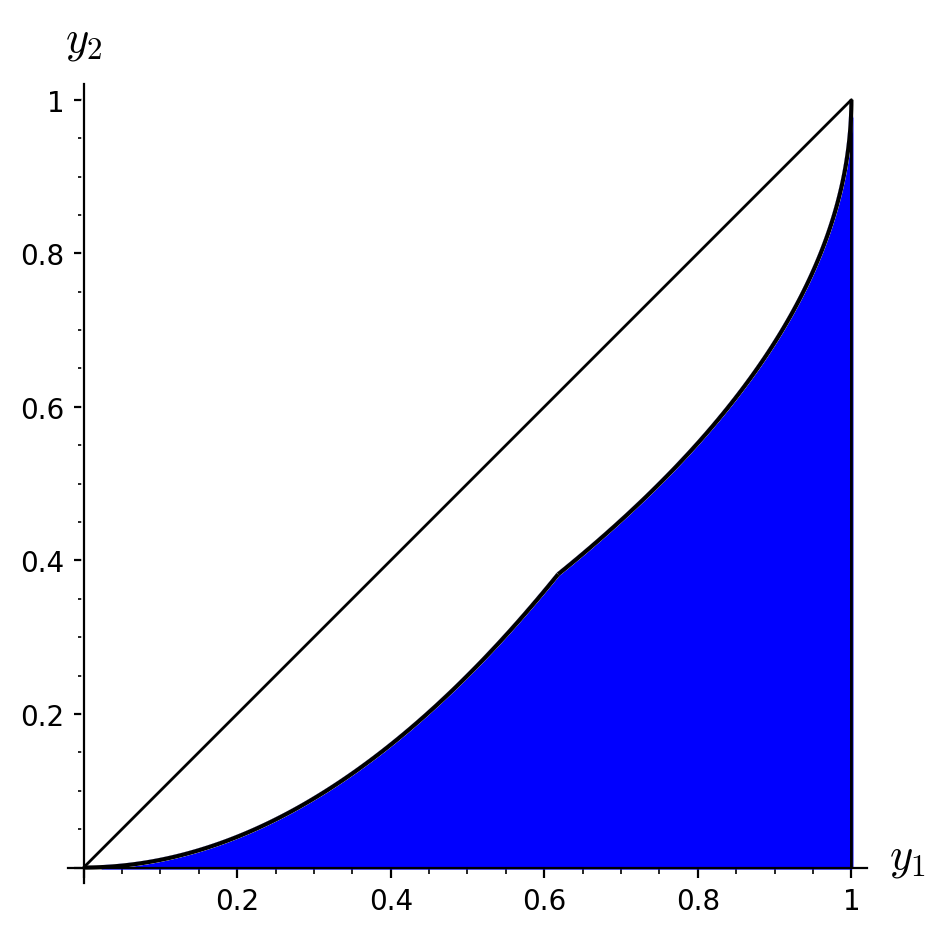}
\caption{The domain $B$ for assumption \ref{enumi:horse2} in Theorem \ref{thm:horseshoes}.}
\end{figure}

The conclusion means that for all $k\in\mathbb{N}$ there exist compact sub-intervals $I_1,\dots,I_k$ of $[0,1]$ with pairwise disjoint interiors, such that $T(I_i)\supset I_j$  for all $i,j\in\{1,\dots, k\}$. We can also get pairwise disjoint intervals, simply by taking $2k$ instead of $k$ and keep every other interval in the family, in the linear order of $[0,1]$, see Section \ref{s:orders}.

It is known \cite{Misiurewicz1979horseshoes} (see also \cite{Ruette2017chaos} and references therein) that having infinite topological entropy is equivalent to having  horseshoes of arbitrary size $k$ in iterates $T^{o(\log k)}$; Theorem \ref{thm:horseshoes} show that some zipper map possess a much stronger property, in that we do not need to iterate them to obtain horseshoes of all order.

The presence of these horseshoes easily imply some ``universality'' properties. 

\begin{coromain}\label{cm:orders}
For all hypersensitive zipper map $T=Z_p$ satisfying either assumptions of Theorem \ref{thm:horseshoes} and all $k,\ell\in\mathbb{N}$, every total order on $k\times\ell$ symbols is realised by $k$ orbits of length $\ell$ of $T$.
\end{coromain}
This means that for all total strict order $\prec$ on the symbols $(s_i^j)_{\substack{0\le j\le \ell \\ 1\le i\le k}}$, there exist $x_1,\dots , x_k\in [0,1]$ such that for all $i,j,i',j'$:
\[ s_i^j \prec s_{i'}^{j'} \Leftrightarrow T^jx_i < T^{j'}x_{i'}.\]

Let $\Omega \se \R$, $S:\Om\to\Om$ and $T:[0,1]\to[0,1]$. We say that $S$ is \emph{embedded} in $T$ (as a dynamical system) if there exists $\pi:\Om\to[0,1]$, injective, such that
\[
\pi\circ S=T\circ \pi.
\]
\begin{coromain}\label{cm:universality}
Let  $T=Z_p$  be a  hypersensitive zipper map satisfying either assumption of Theorem \ref{thm:horseshoes} and $\Om$ a finite set. Then any map $S:\Omega\to\Omega$ is embedded in $T$.
\end{coromain}
The existence of horseshoe of arbitrary order implies that zipper maps to which Theorem \ref{thm:horseshoes} applies have infinite topological entropy, a fact easy to establish directly for all hypersensitive zipper maps. But we can say more, by using a variation of entropy suitable for highly chaotic systems: metric mean dimension, whose definition is recalled in Section \ref{s:meanDim}.
\begin{theomain}\label{t:mdim}
Every hypersensitive zipper map $T=Z_p$ has positive metric mean dimension relative to the Euclidean metric:
\[\mdim(T,\lvert\cdot\rvert) \ge \frac{\log \lambda_{\min}}{\lvert\log h_{\min}\rvert}>0,\]
where $h_{\min} = \min(x_1,x_2-x_1,1-x_2)<1$ and $\lambda_{\min}$ is defined in \eqref{eq:lambdamin}.
\end{theomain}

\begin{coromain}\label{c:KS}
Every hypersensitive zipper map $T=Z_p$ admits an invariant probability measure $\mu$ with positive Kolmogorov-Sinai mean dimension.
\end{coromain}

The \emph{Kolmogorov-Sinai mean dimension} is an invariant of a measured dynamical system $(T,\mu)$ which we introduce in Section \ref{s:meanDim}. Its positivity means that for arbitrarily high $k$, there exist measurable partitions of $[0,1]$ into $k$ subsets for which the entropy grows as a multiple of $\log k$ as $k\to\infty$. Corollary \ref{c:KS} is deduced from Theorem \ref{t:mdim} through an inequality providing one half of a variational principle, Theorem \ref{t:variational}: the existence of a metric of finite dimension for which the relative metric mean dimension is positive enables the construction of measures of positive Kolmogorov-Sinai metric mean dimension. 
However, unlike the classical variational principle, the complexity of these measures does not bound below the metric mean dimension of arbitrary metrics, even when controlling their dimension:
\begin{theomain}\label{t:vanishing}
Every continuous map $T:[0,1]\to[0,1]$ has zero absolute metric mean dimension; more precisely, there exist a metric $d$ on $[0,1]$, inducing the usual topology, such that $\mdim(T,d) =0$ and $\udimM([0,1],d)=1$.
\end{theomain}
Here $\udimM$ denotes the upper Minkowski dimension; it bounds from above the Hausdorff dimension, which must thus also be equal to $1$. This result shows a small case of the conjecture (see e.g. \cite{lindenstrauss-Tsukamoto2019double}) that mean dimension equals the infimum over all metrics of the metric mean dimension.

\subsection{Further directions}

\subsubsection{Further properties of zipper maps}

We have left many questions open even when restricting to the case of zipper maps. One could first determine the exact range of Theorem \ref{thm:horseshoes}, by determining the set of all values of the parameter $p$ for which $T=Z_p$ admits horseshoes of all order.

Another intriguing question is that of topological conjugacy: for which $p,p'$ are the zipper maps $Z_p$, $Z_{p'}$ topologically conjugate? Is there a continuum of conjugacy classes among the $Z_p$? Note that one does not hope for any sort of topological stability here: the numerous local extrema can change their relative positions even under arbitrary small perturbations of $p$.

Zipper maps have numerous invariant measures, e.g. when Theorem \ref{thm:horseshoes} applies one can consider any invariant measure of a finite shift $\{1,\dots, k\}^\mathbb{N}$ and use a horseshoe of order $k$ to build from it an invariant measure for $Z_p$. They are however localized in a tiny part of the space and only witness a tiny amount of the dynamics at play. Is there a way to single out a ``most natural'' measure? How many physical measures does a given or typical zipper map possess? Is there something alike a measure of maximal entropy (which shall be defined in a way to be specified, more restrictively than only asking for the Kolmogorov-Sinai entropy to be infinite, for otherwise there would be no hope for uniqueness).

\subsubsection{Rough dynamics for other specific maps}

While zipper maps already provide an interesting playground, the original question that sparked this work was about the more classical Weierstrass functions such as
\[W_{a,b}(x) = \sum_{n\in\mathbb{N}} a^n \cos(2\pi b^n x)\]
with $ab>1$, $a<1$. Letting $I$ be the range of $W_{a,b}$, what are the dynamical properties of the map $T=W_{a,b}$ from $I$ to itself? While Weierstrass appear as \emph{function} in several works of dynamical flavor, they do not seem to have been considered as \emph{maps} to be iterated.

Another class of irregular maps that one could study is the graph of Brownian motion. Brownian motion $(B_t)_{t\in\mathbb{R}}$ is a continuous-time stochastic process, i.e. it is constructed as a measurable map $\omega\mapsto (B_t(\omega))_{t\in\mathbb{R}}$ from a standard probability space $(\Omega,\mathscr{F},\mathbb{P})$ to the space of continuous functions $\mathbb{R}\to\mathbb{R}$. In this definition, the domain $\mathbb{R}$ is thought as time and the range $\mathbb{R}$ as space; but for almost any event $\omega\in\Omega$, we can consider the map $T:t\mapsto B_t(\omega)$. It is well known that $\lvert B_t(\omega)-B_s(\omega)\rvert$ is very roughly of the order of $\lvert t-s\rvert^{\frac12}$, in particular much larger than $\lvert t-s\rvert$ when $t,s$ are close, but much smaller than $\lvert t-s\rvert$ when they are far away. The map $T$ is thus almost surely locally $\alpha$-H\"older for all $\alpha<\frac12$, but far from Lipschitz. What are its \emph{typical} (e.g. almost sure) dynamical properties? With the usual normalization, $T(0)=0$ and $T(t)=o(t)$, so that there is a compact interval of non-empty interior $I$ such that $T(I)= I$, and every orbit of $T$ is attracted to $I$ extremely fast. One can thus restrict to study the dynamics of $T$ on the random interval $I$.

\subsubsection{Rough dynamics for generic maps}

In \cite{Yano1980homeo}, Yano proved that a $C^0$-generic continuous map of a manifold has infinite entropy; this was refined in term of metric mean dimension in \cite{Carvalo+2022generic}. A general theory of rough dynamics should include the answer to the following problems:
\begin{itemize}
\item find other ``hyperchaotic'' dynamical properties of $C^0$-generic continuous maps of manifolds,
\item what happens for generic $\alpha$-Hölder maps, $\alpha\in(0,1)$? 
\end{itemize}
In both cases, the presence of horseshoes of arbitrary order is likely to be easy to prove.

\subsubsection{First-order logic of continuous interval maps}

We would like to suggest the study of first-order logic of continuous interval maps in a similar way to first-order logic of groups. Let us briefly recall the case of groups first. 

One considers sentences build from logical connectives, quantifiers, the group operations $*$ and $\null^{-1}$ and variable names to be interpreted as element of a particular group. For example, a sentence such as $x*y=y*x$ expresses that two elements $x,y$ commute, while the closed sentence\footnote{I.e. sentence without free variables, all variables being introduced by a quantifier.} $\forall x,y, x*y=y*x$ expresses commutativity of the group. The \emph{elementary theory} of a group $G$ is the set of closed sentences that are true in $G$, and two groups are called \emph{elementary equivalent} when they have the same elementary theory. Observe that variables are to be interpreted solely as group elements, so that it is not allowed to quantify on other type of variables such as integers.
Fueled by Tarski's problem, the first-order logic of groups has played an important role in relation to group theory and geometric group theory, see e.g. \cite{Sela2006diophantine,Perin2008plongements}.
Typical questions revolve around algebraic and geometric properties of groups that are preserved under elementary equivalence.

Corollary \ref{cm:orders} hints that a similar theory could be developed for continuous maps of the interval $[0,1]$, by including as symbols usable in sentences: $T$ to denote application of the map, and $<$ for the usual order. First-order sentences are well-formed sentences in the symbols 
$\{T,<,=, \mathrm{or},\mathrm{and}, \mathrm{not}, (,), \forall,\exists\} \cup V$ where $V=\{x,y,z,\dots\}$ is a countably infinite set of variable names to be interpreted as taking value in $[0,1]$. For example,
\[\exists x, (T(x)\neq x \text{ and }T^2(x)=x)\]
is a first-order sentence expressing the existence of an orbit of period $2$. However the sentence
\[\exists z, \forall x,y, \exists n\in\mathbb{N}, x\ge y \text{ or } x<T^n(z)<y\]
expressing topological transitivity is not a valid first-order sentence, since there is a quantified variable taking values in $\mathbb{N}$.

To each continuous map $T:[0,1]\to[0,1]$ is thus associated its elementary theory, the set of closed sentences it satisfies. The question is then: which dynamical properties of $T$ are shared by all maps elementary equivalent to $T$? 
Corollary  \ref{cm:orders} says that all maps $Z_p$ to which Theorem \ref{thm:horseshoes} applies share the part of the elementary theory where only the \emph{existential} quantifier is used. Are they all elementary equivalent? What about Baire-generic continuous maps of the interval? What about classical maps, e.g. the quadratic family? Does the elementary theory identify the topological conjugacy class inside such a family? Are properties such as topological transitivity invariant under elementary equivalence?

As an example of the dynamical relevance of first-order logic, it is left as an exercise to the reader to prove that if two continuous interval maps $T, S : [0,1]\to[0,1]$ are elementary equivalent, then they have the same topological entropy. Hints: we only need that they satisfy the same existential sentences; do not try to translate $h(T)=y$ as a single sentence; use Misiurewicz' theorem according to which entropy is arbitrary close to be realised by horseshoes in iterates of the map \cite{Misiurewicz1979horseshoes}.

\section{Basic properties of zipper maps}

\subsection{Additional notations}

There are several Iterated Function Systems hiding behind the contraction $\Phi_p$, that will be useful in various arguments. First, there are the horizontal and vertical IFS on $[0,1]$:
\begin{align*}
H_0(x) &= x_1\cdot x & V_0(y) &= y_1\cdot y \\
H_1(x) &= (x_2-x_1)x+x_1 & V_1(y) &= (y_2-y_1)y+y_1 \\
H_2(x) &= (1-x_2)x+x_2 & V_2(y) &= (1-y_2)y+y_2
\end{align*}
that will be used to define relevant sub-intervals of $[0,1]$, and we can combine them into the IFS on the square $[0,1]^2$ defined by the three contractions
$P_i(x,y) = \big(H_i(x), V_i(y)\big)$, whose attractor is the graph of $Z_p$.
A word with letters in the alphabet $\{0,1,2\}$ will be written under the form $\omega = i_1i_2\dots i_\ell$, where $\ell\in\mathbb{N}$ is called its \emph{length} and is also denoted by $\lvert \omega\rvert$. We use exponents to denote repetition of a letter, e.g. $0^k = 0\cdots 0$ ($k$ times). The set of words of length $n$ is denoted by $\{0,1,2\}^n$, and the set of finite words by $\{0,1,2\}^*$. Given $\omega = i_1i_2\dots i_\ell\in\{0,1,2\}^*$ we set $H_\omega = H_{i_1}\circ H_{i_2} \circ\dots\circ H_{i_\ell}$ and we define similarly $V_\omega$ and $P_\omega$. 

We define intervals $I_\omega = H_\omega([0,1])$ and $J_\omega = V_{\omega}([0,1])$ (in particular $I_0 = [0,x_1]$, $I_1 = [x_1,x_2]$ and $I_2 = [x_2,1]$; $J_0 = [0,y_1]$, $J_1 = [y_2,y_1]$ and $J_2 = [y_2,1]$); they satisfy the relations 
\begin{align*}
Z_p(I_\omega) &= J_\omega \\
I_\omega \subset I_\sigma &\text{ and } J_\omega \subset J_\sigma \quad\text{whenever $\sigma$ is a prefix of $\omega$}
\end{align*}
and for each $\ell$, the family $(I_\omega)_{\lvert \omega\rvert=\ell}$ is a \emph{tiling} of $[0,1]$ (the intervals have disjoint interiors and their union is  $[0,1]$) while the family $(J_\omega)_{\lvert \omega\rvert=\ell}$ is a covering of $[0,1]$. Given a word $\omega=i_1 i_2\dots i_\ell$, we will sometimes use the word \emph{parent} to mean the word $\omega'=i_1 i_2\dots i_{\ell-1}$, i.e. the largest strict prefix. Similarly, the parent of an interval $I_\omega$ or $J_\omega$ is the corresponding interval $I_{\omega'}$ or $J_{\omega'}$.
These intervals combine into rectangles $R_\omega = I_\omega \times J_\omega = P_\omega([0,1]^2)$, whose diagonals have slopes at least $\lambda_{\min}^{\lvert\omega\rvert}$. Observe that since $\lambda_{\min}>1$, the top-right vertex $(x_1,y_1)$ of $R_0$ lies above the diagonal and the bottom-left vertex $(x_2,y_2)$ of $R_2$ lies below the diagonal.

We introduce the following quantities (again implicitly depending upon $p$):
\begin{align*}
v_{\min} &=\min(y_1,y_1-y_2,1-y_2) & \\
h_{\min} &= \min(x_1,x_2-x_1,1-x_2) & h_{\max} &= \max(x_1,x_2-x_1,1-x_2)
\end{align*}
i.e. $v_{\min}$ is the smallest height of the rectangles $R_0,R_1,R_2$ while $h_{\min}$ is their smallest width, $h_{\max}$ their maximal width. For all $\alpha\in(0,1]$ we also set:
\begin{gather*}
\lambda_0(\alpha) = \frac{y_1}{x_1^\alpha} \qquad \lambda_1(\alpha) = \frac{y_1-y_2}{(x_2-x_1)^\alpha} \qquad \lambda_2(\alpha) = \frac{1-y_2}{(1-x_2)^\alpha} \\
\lambda_{\min}(\alpha) = \min\big(\lambda_0(\alpha),\lambda_1(\alpha),\lambda_2(\alpha)\big) \qquad \lambda_{\max}(\alpha) = \max\big(\lambda_0(\alpha),\lambda_1(\alpha),\lambda_2(\alpha)\big) 
\end{gather*}
In particular $\lambda_{\min}=\lambda_{\min}(1)$.

\subsection{Regularity}

To study the regularity of zipper maps, we first search for conditions on $p$ ensuring that $\Phi_p$ preserves $\Ck^\alpha_0$ and contracts the $\alpha$-H\"older semi-norm (up to a constant). 
\begin{lemm}
For all $f\in\Ck^\alpha_0$, $\Hol_\alpha(\Phi_p f) \le \max\big( \lambda_{\max}(\alpha) \Hol_\alpha(f), (x_2-x_1)^{-\alpha} \big)$.
\end{lemm}

\begin{proof}
Let $x < x'$ be two points on $[0,1]$.
If both lie in the same $I_j$ for some $j\in\{0,1,2\}$, a direct computation yields
\[\lvert \Phi_p f(x)-\Phi_p f(x') \rvert \le \lambda_j(\alpha) \Hol_\alpha(f) \lvert x-x'\rvert^\alpha \le \lambda_{\max}(\alpha) \Hol_\alpha(f) \lvert x-x'\rvert^\alpha.
\]
When $x\in I_0$ and $x'\in I_1$, using $\Phi_p f(I_j) = J_j$ it comes
\begin{align*}
\lvert \Phi_p f(x)-\Phi_p f(x') \rvert &\le \max(\lvert \Phi_p f(x) -y_1\rvert, \lvert y_1- \Phi_p f(x') \rvert) \\
  &\le \lambda_{\max}(\alpha) \Hol_\alpha(f) \max(\lvert x-x_1\rvert^\alpha,\lvert x'-x_1\rvert^\alpha) \\
  &\le \lambda_{\max}(\alpha) \Hol_\alpha(f) \lvert x-x'\rvert^\alpha.
\end{align*}
The same argument applies when $I_0,I_1$ are replaced by $I_1,I_2$.
Last, if $x\in I_0$ and $x'\in I_2$, we have $\lvert x'-x \rvert \ge x_2-x_1$ and therefore
\begin{align*}
\lvert \Phi_p f(x)-\Phi_p f(x') \rvert &\le 1 \\
  &\le \frac{\lvert x'-x \rvert^\alpha}{(x_2-x_1)^\alpha}.
\end{align*}
\end{proof}

\begin{prop}\label{prop:holder}
The zipper map $T=Z_p$ is $\alpha_{\min}$-H\"older where
\[\alpha_{\min} = \max \{\alpha \mid \lambda_{\max}(\alpha)\le 1\} = \min\Big(\frac{\log y_1}{\log x_1}, \frac{\log(y_1-y_2)}{\log(x_2-x_1)}, \frac{\log(1-y_2)}{\log(1-x_2)} \Big)>0.\]
\end{prop}

\begin{proof}
Fix $\alpha=\alpha_{\min}$ and let $f_k = \Phi_p^k(\Id_{[0,1]})$: we have $f_k\to T$ in the uniform norm and $\lambda_{\max}(\alpha)=1$. The previous Lemma and an induction yields for all $k$ that $\Hol_{\alpha}(f_k) \le \max\big(\Hol_{\alpha}(\Id_{[0,1]}),(x_2-x_1)^{-\alpha}\big)=(x_2-x_1)^{-\alpha}$ and is thus bounded independently of $k$. It follows that $T$ is $\alpha$-H\"older with $\Hol_{\alpha}(T)\le (x_2-x_1)^{-\alpha}$.
\end{proof}

\subsection{Roughness}

We now want to express an ``irregularity''; we cannot hope for an inequality of the form $c\lvert x-x'\rvert^\beta \le \lvert T(x)-T(x')\rvert$ since $T=Z_p$ is not even locally one-to-one. We introduce the following terminology.
\begin{defi}
A continuous map $f:I\to J$ between intervals is said to be \emph{$\beta$-hypersensitive} if there exists $C>0$ such that for all interval $A\subset I$ it holds
\[\lvert f(A) \rvert \ge C\lvert A\rvert^\beta\]
(where $\lvert\cdot\rvert$ denotes the length of an interval). A map is said to be \emph{hypersensitive} if it is $\beta$-hypersensitive for some $\beta\in(0,1)$.
\end{defi}
Hypersensitivity is a kind of strengthening of the property of being expanding suitable for irregular, not locally one-to-one maps: small intervals grow super-exponentially in size when $f$ is applied repeatedly. This choice of terminology reconciles with the use of the word in the introduction:
\begin{prop}\label{p:hypersensitivity}
The zipper map $T=Z_p$ is hypersensitive if and only if $\lambda_{\min}>1$ (with exponent $\beta=1+\frac{\log \lambda_{\min}}{\log h_{\min}} \in(0,1)$).
\end{prop}

\begin{proof}
If $\lambda_{\min} \le 1$, let $i\in \{0,1,2\}$ be such that $R_i$ has diagonals of slopes at most $1$ and consider the word $i^\ell=ii \cdots i$ of length $\ell\in\mathbb{N}$. Then $\lvert Z_p(I_{i^\ell})\rvert = \lvert J_{i^\ell} \rvert \le \lvert I_{i^\ell}\rvert$
while $\lvert I_{i^\ell}\rvert$ can be made arbitrarily small by taking $\ell$ large enough. This prevents $Z_p$ from being hypersensitive.

Assume now $\lambda_{\min} > 1$.
Let $A\subset [0,1]$ be an interval, and let $k$ be the minimal positive integer such that for some word $\omega$ of length $k$, $I_\omega\subset A$; since the length of $I_\omega$ is between $h_{\min}^k$ and $h_{\max}^k$, such a $k$ must exist and we have 
\[ \frac{\log\lvert A\rvert}{\log h_{\min}} \le k \le \frac{\log\frac{\lvert A\rvert}{2}}{\log h_{\max}} + 1\]
(if the second inequality did not hold, the tiling of depth $k-1$ would have all its elements of size less than $\frac{\lvert A\rvert}{2}$, so that at least one of them would be a sub-interval of $A$, contradicting the minimality of $k$).
The same argument yields $\lvert I_\omega \rvert \ge \frac{h_{\min}}{2} \lvert A\rvert$, otherwise either the parent of $I_\omega$ or a neighbor of its parent would be contained in  $A$.

We obtain
\[
\lvert Z_p(A) \rvert \ge \lvert Z_p(I_\omega)\rvert = \lvert J_\omega \rvert \ge
 \lambda_{\min}^k \lvert I_\omega\rvert 
\ge \lambda_{\min}^{\frac{\log\lvert A\rvert}{\log h_{\min}}}\frac{h_{\min}}{2} \lvert A\rvert
 \ge \frac{h_{\min}}{2} \lvert A\rvert^{1+\frac{\log \lambda_{\min}}{\log h_{\min}}}
\]
\end{proof}

\section{Horseshoes of arbitrary order}

In this section we prove Theorem \ref{thm:horseshoes}, then deduce Corollaries \ref{cm:orders} and \ref{cm:universality}.

\subsection{Symmetric parameters}

We start with the case when $T=Z_p$ satisfies assumption \ref{enumi:horse1} in Theorem \ref{thm:horseshoes}. This is the easiest case, while providing a good introduction to the second case.

Since the parameter $p$ is symmetric, $\frac12$ is a fixed point of $T$. The idea is to zoom in to this point and use hypersensitivity to find many thin and high rectangles $(R_\omega)_{\omega\in A}$ in the vicinity of $(\frac12,\frac12)$.

\begin{lemm}
Assume $p$ is symmetric (assumption \ref{enumi:horse1} in Theorem \ref{thm:horseshoes}). Then for all $k\in\mathbb{N}$, there exist a set $A$ of words on the alphabet $\{0,1,2\}$, with cardinal at least $k$, such that for all $\omega,\sigma\in A$:
\begin{itemize}
\item $I_\omega$ and $I_\sigma$ have disjoint interior whenever $\omega\neq\sigma$,
\item $J_\omega \supset I_\sigma$.
\end{itemize}
\end{lemm}

Case \ref{enumi:horse1} of Theorem \ref{thm:horseshoes} follows right away since
$T(I_\omega)=J_\omega\supset I_\sigma$ for all $\omega,\sigma\in A$. 

An interpretation of the proof below which will prove useful in the sequel is that we shall use $P_{\omega_0}^{-1}$ as a zoom-in map, which renormalises a small (rather tall and very thin) rectangle $R_{\omega_0}$ into the unit square, sending the diagonal of $[0,1]^2$ to an almost horizontal line. That line is the graph of the identity map in ``local coordinates'' provided by $P_{\omega_0}$; in those coordinates, the identity map has a certain ``range'' (contained in $[\frac12-\eta,\frac12+\eta]$ below). We then only have to choose $\omega_0$ in a way ensuring that many rectangles $R_{\omega_+}$ have their vertical sides cover this ``range'' of the identity map; the words of the form $\omega_0\omega_+$  will constitute $A$.

\begin{proof}
For each $n\in\mathbb{N}$ we consider the word $1^n$ (the word $11\dots1$ of length $n$) and consider $P_{1^n}^{-1} : R_{1^n}\to R=[0,1]^2$. The image $D_{1^n} = P_{1^n}^{-1}(\{(x,y)\in R_{1^n} \mid x=y\})$ of the diagonal is a line containing the point $(\frac12,\frac12)$ and of slope $\lambda_1^{-n}$ where $\lambda_1=\frac{y_1-y_2}{x_2-x_1}>1$. We will choose $n$ and a $\eta>0$ ensuring two conditions, the first one being that $D_{1^n}$ is contained in $[0,1]\times (\frac12-\eta,\frac12+\eta)$ (which is equivalent to $H_{1^n}([0,1])\subset V_{1^n}((\frac12-\eta,\frac12+\eta))$); for this it is sufficient to have
\begin{equation}
\eta > \frac12\Big(\frac{x_2-x_1}{y_1-y_2}\Big)^n \, .
\label{e:delta1}
\end{equation}

The second condition we want to ensure is that $J_{1^{\ell} i}\supset (\frac12-\eta,\frac12+\eta)$ for all $\ell\le k$ and $i \in \lrs{0,1,2}$. By writing $J_{1^{\ell} i}=V_1^{\ell}(J_i)$ and using $V_1(\frac12)=\frac12$ and $(y_2,y_1) \subset J_i$, we see that for this condition to hold it is sufficient to have
\begin{equation}
\eta < \big(\frac12-y_2\big)(y_1-y_2)^{k}
\label{e:delta2}
\end{equation}
By choosing $\eta$ small enough and then $n$ large enough, we can ensure that both \eqref{e:delta1} and \eqref{e:delta2} are satisfied. We then set $A=\big\{1^{n+\ell} i \mid 0 \le \ell\le k, \ell \text{ even},  i\in\{0,2\}\big\}$.

Given $\omega\neq\sigma\in A$, the intervals $I_\omega$ and $I_\sigma$ have disjoint interior since $\omega$ and $\sigma$ are not prefixes one of the other.
Last, for all $\omega=1^{n+\ell} i$ and $\sigma=1^{n+m} j$ in $A$:
\[ I_\sigma = H_{1^n}(I_{1^m j})\subset H_{1^n}([0,1])\subset V_{1^n}\lr{\lr{\frac12-\eta,\frac12+\eta}} \subset V_{1^n}(J_{1^\ell i}) = J_\omega. \]
\end{proof}

To treat the second case of Theorem \ref{thm:horseshoes}, the difficulty will be that we do not have the fixed point $(\frac12,\frac12)$ as reference. We will have to ensure that the almost horizontal line $D_{\omega_0}$ stays far from the upper and lower sides of the unit square, as otherwise the ``range'' of the identity in our local coordinates will be impossible to cover with intervals $J_{\omega_+}$.

\subsection{An open region of the parameter space}

We turn to the case when $T=Z_p$ satisfies assumption \ref{enumi:horse2} in Theorem \ref{thm:horseshoes}.

We consider for each $(y_1,y_2)$ the affine Iterated Function System $(V_\omega)_{\omega\in\{0,2\}}$ generated by $V_0$ and $V_2$ as defined above (dependence on $(y_1,y_2)$ kept implicit), which is related to (asymmetric) Bernoulli convolutions.

Recall that $B = \{(y_1,y_2)\in(0,1)^2 \mid y_1^2>y_2 \text{ and } y_1>(2-y_2)y_2 \}$, i.e. $B$ is the set of parameters such that $V_{00}(1)>V_2(0)$ and $V_{22}(0)<V_0(1)$. 


\subsubsection{Properties of the vertical IFS}

We start with some basic properties of $(V_\omega)_{\omega\in\{0,2\}^*}$, assuming $(y_1,y_2)\in B$.
We consider the functional $S$ acting on the space of bounded functions $\mathbb{R}\to\mathbb{R}$ by
\[Sf(x) := f\circ V_0^{-1}(x) + f\circ V_2^{-1}(x) = f\Big(\frac{x}{y_1}\Big) + f\Big(\frac{x-y_2}{1-y_2}\Big)\]
and we denote by $\one$ the characteristic function of the interval $[0,1]$, and by $\one_I$ the characteristic function of an interval $I$. Observe that $S\one_I = \one_{V_0(I)}+\one_{V_2(I)}$ so that for all $n\in\mathbb{N}$ and $x\in[0,1]$, $S^n\one(x)$ is the number of words $\omega\in\{0,2\}^*$ of length $n$ such that $x\in V_\omega([0,1])$.

\begin{lemm}
For all $\varepsilon>0$, there exist $\delta>0$ and $n\in\mathbb{N}$ such that $S^n\one_{[\delta,1-\delta]}\ge 2\one_{[\varepsilon,1-\varepsilon]}$.
\end{lemm}

\begin{proof}
Let $a,b\in(0,1)$ such that $a<b$. Whenever $a$ is small enough and $b$ large enough, more precisely whenever $V_2(a)<V_0(b)$, we have:
\[S\one_{[a,b]} = \one_{[V_0(a),V_0(b)]} + \one_{[V_2(a),V_2(b)]} = \one_{[V_0(a),V_2(b)]} + \one_{[V_2(a),V_0(b)]}.\]
Observe that $V_0(a)<a$ and $V_2(b)>b$, so we can apply this to $\one_{[V_0(a),V_2(b)]}$ to compute
\[S^2\one_{[a,b]} = \one_{[V_{00}(a),V_{22}(b)]} + \one_{[V_{20}(a),V_{02}(b)]} + \one_{[V_{02}(a),V_{00}(b)]} + \one_{[V_{22}(a),V_{20}(b)]}.\]
Since $(y_1,y_2)\in B$, $V_{00}(1)>V_2(0) = V_{20}(0)$ and $V_{22}(0)< V_0(1) = V_{02}(1)$. For $a$ small enough and $b$ large enough we thus have
$V_{00}(b)>V_{20}(a)$ and $V_{22}(a)< V_{02}(b)$, from which it follows
\[S^2\one_{[a,b]} \ge \one_{[V_{00}(a),V_{22}(b)]} + \one_{[V_{02}(a),V_{20}(b)]}.\]
Fix $a$ and $b$ satisfying all above hypotheses: $V_2(a)<V_0(b)$, $V_{00}(b)>V_{20}(a)$ and $V_{22}(a)< V_{02}(b)$. Using that for all $x\in(0,1)$, $V_2(x)>x$ and $V_0(x)<x$, for all $n$ we have
\[V_{02^n0}(b) > V_{00}(b) >V_{20}(a) >V_{20^{n+1}}(a)\]
and similarly $V_{20^n2}(a)<V_{02^{n+1}}(b)$.
This is exactly what is needed to prove by induction that for all $n\in\mathbb{N}$,
\[S^{n+1}\one_{[a,b]} \ge \one_{[V_{0^{n+1}}(a),V_{2^{n+1}}(b)]} + \one_{[V_{0^n2}(a),V_{2^n0}(b)]}.\]

Applying this to $a=\delta$ and $b=1-\delta$ for $\delta>0$ small enough and using $\lim_n V_{0^{n+1}}(\delta)=\lim_n V_{0^n2}(\delta)=0$ and $\lim_n V_{2^{n+1}}(1-\delta) = \lim_n V_{2^n0}(1-\delta)=1$, we obtain the desired conclusion.
\end{proof}

\begin{lemm}\label{l:cover}
For all $\varepsilon>0$ and all $k>0$, there exist $\delta>0$ and $n\in\mathbb{N}$ such that $S^n\one_{[\delta,1-\delta]}\ge k\one_{[\varepsilon,1-\varepsilon]}$.
\end{lemm}

\begin{proof}
Fix $\varepsilon>0$ and $k>0$. The previous lemma permits for any $\ell\in\mathbb{N}$ to find $\delta_1,\dots,\delta_\ell>0$ and $n_1,\dots, n_\ell\in\mathbb{N}$ such that 
$S^{n_1}\one_{[\delta_1,1-\delta_1]} \ge 2\one_{[\varepsilon,1-\varepsilon]}$ and for all $j\in\{2, \dots,\ell\}$,
\[S^{n_j}\one_{[\delta_{j},1-\delta_{j}]} \ge 2\one_{[\delta_{j-1},1-\delta_{j-1}]}.\]
It is then sufficient to choose $\ell$ such that $2^\ell\ge k$ and set $n=n_1+\dots+n_\ell$ and $\delta=\delta_\ell$.
\end{proof}

Before we deduce the result we shall need in the proof of Theorem \ref{thm:horseshoes}, let us note that an experimental exploration indicates that the set of parameters $(y_1,y_2)$ for which the conclusion of Lemma \ref{l:cover} holds is larger than $B$, but smaller than the triangle $\{y_1>y_2\in(0,1)\}$. Our proof does not use $(y_1,y_2)\in B$ behind this point, so we could replace condition \ref{enumi:horse2} in Theorem \ref{thm:horseshoes} by the conclusion of Lemma \ref{l:cover}; this would be more general, but less explicit.

\begin{prop}\label{p:cover}
For all $\varepsilon>0$ and all $k>0$, there exist $\eta>0$ and $n\in\mathbb{N}$ such that for all $y\in[\varepsilon,1-\varepsilon]$, 
\[\#\big\{\omega\in\{0,2\}^n \,\big|\, [y-\eta,y+\eta]\subset V_\omega([0,1])\big\}\ge k.\]
\end{prop}

\begin{proof}
Let $\varepsilon>0$ and use Lemma \ref{l:cover} to find $n\in\mathbb{N}$ and $\delta>0$ such that $S^n\one_{[\delta,1-\delta]}\ge k\one_{[\varepsilon,1-\varepsilon]}$. Let $v=\min(y_1,1-y_2)$ and $\eta=\delta v^n$. Let $y\in[\varepsilon,1-\varepsilon]$; then there are at least $k$ words $\omega\in \{0,2\}^n$ such that $y\in V_\omega([\delta,1-\delta])$. For each such $\omega$, we have $V_\omega(0)\le V_\omega(\delta)-\delta v^n\le y-\eta$ and similarly $y+\eta\le V_\omega(1)$, so that $[y-\eta,y+\eta]\subset V_\omega([0,1])$.
\end{proof}

%
%
%

\subsubsection{Finding horseshoes}

Given $\varepsilon>0$, a line $L$  is said to be $\varepsilon$-\emph{transverse} to $R=[0,1]^2$ whenever $\varnothing\neq L\cap R \subset [0,1]\times [\varepsilon,1-\varepsilon]$. $L$ is said to be $\varepsilon$-\emph{transverse} to $R_\omega$ if $P_\omega^{-1}(L)$ is $\varepsilon$-transverse to $R$. Writing $R_\omega=[a,b]\times [c,d]$, this condition is equivalent to $\varnothing\neq L\cap R_\omega\subset [a,b]\times [c+\varepsilon(d-c),d-\varepsilon(d-c)]$. This definition is illustrated in Figure \ref{f:transverse}.
\begin{figure}[htp]
\centering
\includegraphics[width=.4\linewidth]{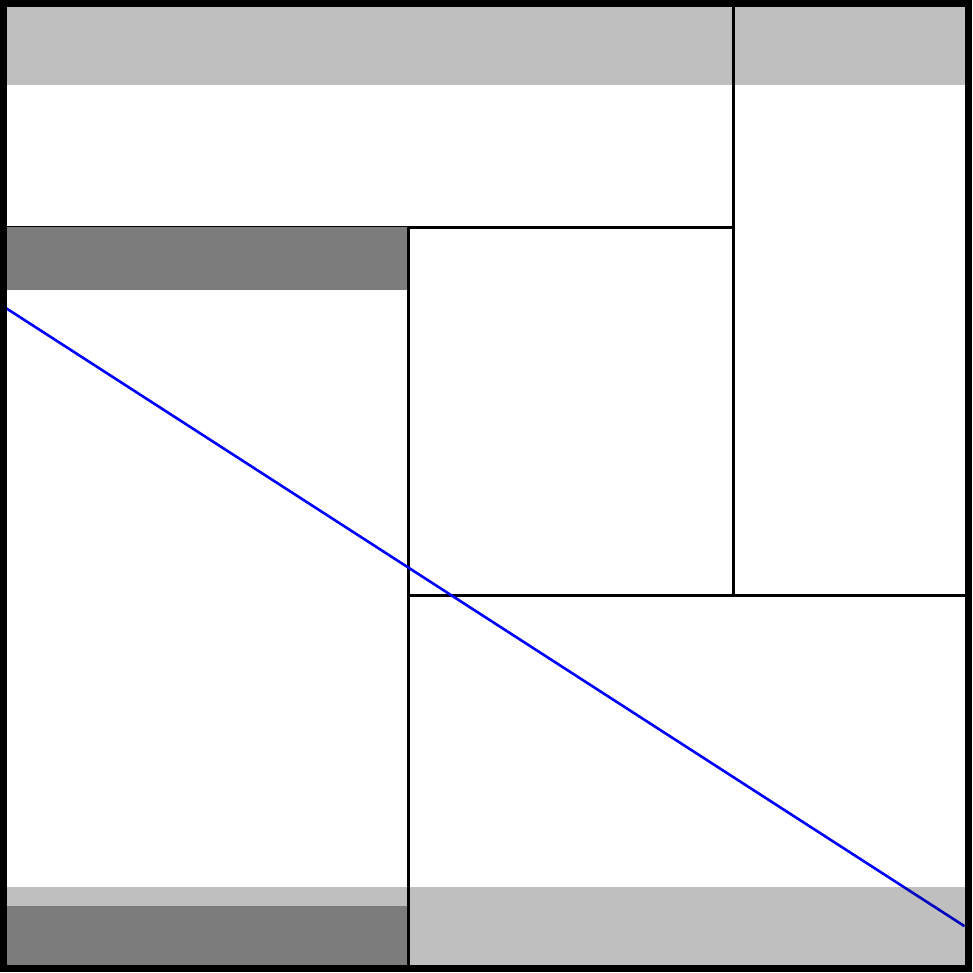}
\caption{If $\varepsilon$ is the height of the light grey bands, then the blue line is not $\varepsilon$-transverse to $R$, but is $\varepsilon$-transverse to $R_0$. }\label{f:transverse}
\end{figure}

We prove Theorem \ref{thm:horseshoes} under condition \ref{enumi:horse2} in two steps, the first stated as a lemma. The general strategy is to zoom in using $P_\omega^{-1}$ to replace the diagonal by an almost horizontal, somewhat vertically centered line, as illustrated by Figure \ref{f:zoomins}. 
\begin{figure}[htp]
\centering
\includegraphics[width=.8\linewidth]{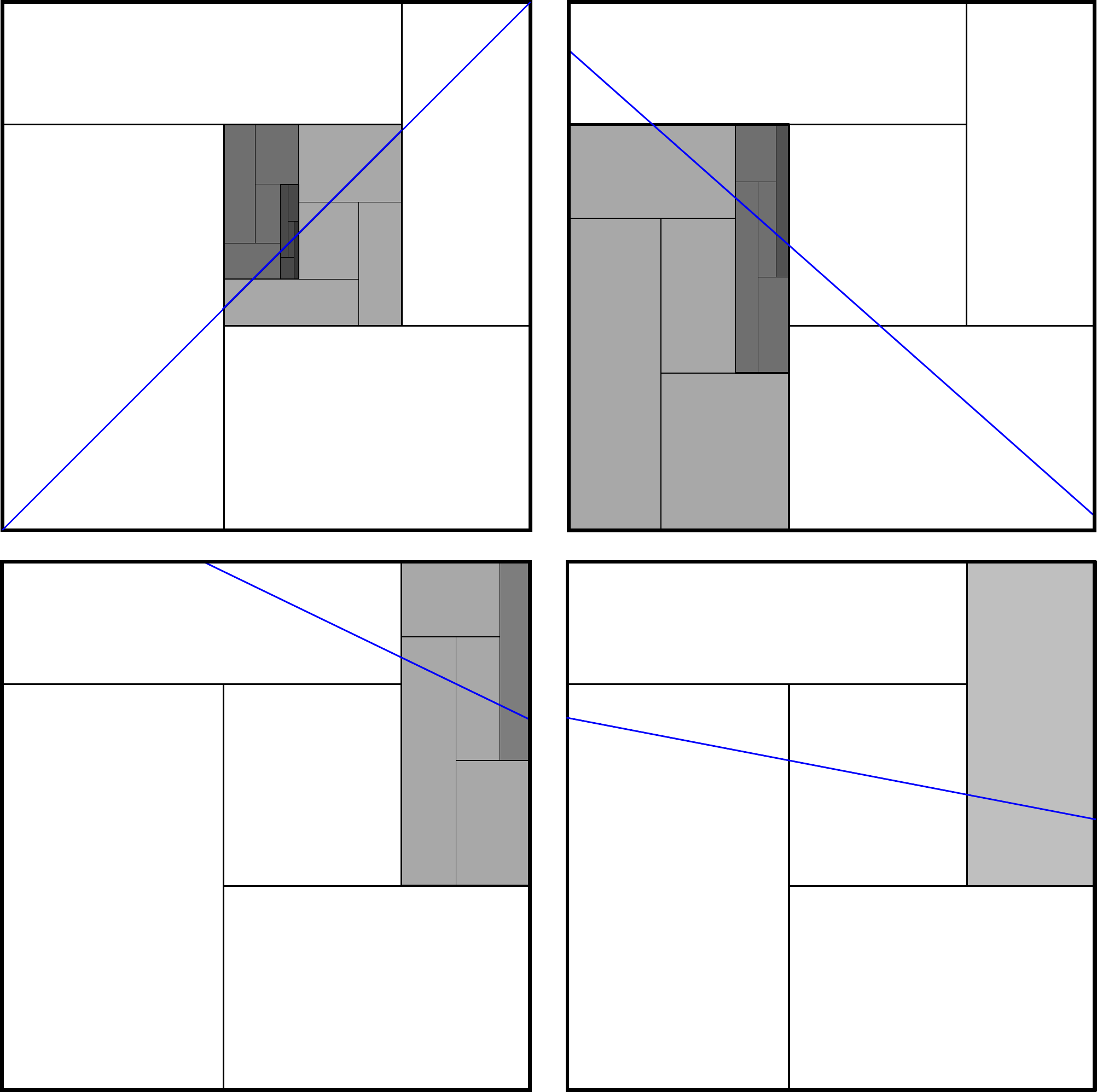}
\caption{Upper-left: the square $R$ and the diagonal $D$ (blue), with successive $R_\omega$ we shall zoom in to in shades of grey. Each successive picture is an affine zoom-in of the light grey area in the previous one. The choice of the letter is done to ensure good transversality. }\label{f:zoomins}
\end{figure}

\begin{lemm}\label{l:transverse}
There exist $\varepsilon=\varepsilon(p)>0$ and $\ell_0=\ell_0(p)\in\mathbb{N}$ such that for all $\ell\ge \ell_0$, one can find a word $\omega$ of length $\ell$ such that the diagonal $D=\{(x,x):x\in\mathbb{R}\}$ is $\varepsilon$-transverse to $R_\omega$.
\end{lemm}

\begin{proof}
The strategy of proof
is to ``zoom in'' using some $P_{\omega_0}^{-1}$ to get an almost horizontal line transversal to $R$. Then, using almost horizontality, it will be easy to expand $\omega_0$ while preserving transversality.

Given any word $\omega$, we set $D_\omega = P_\omega^{-1}(D)$ (note that $P_\omega$ extends to an affine bijection of $\mathbb{R}^2$, and whenever needed $P_\omega^{-1}$ denotes the inverse of this extension). For some $\omega$, $D_\omega\cap R=\varnothing$ but we will choose our words so that these lines are not considered. Observe that the slope of $D_\omega$ is at most $\lambda_{\min}^{-\lvert\omega\rvert}$ in absolute value.

Let $\mathcal{L}$ be the set of lines meeting the interior of at least one of the vertical sides of $R$, and of slope in $(-1,0)$. Since $\lambda_{\min}>1$, the diagonal must meet the interior of at least one of the vertical sides of the central rectangle $R_1$, thus $D_1\in\mathcal{L}$.

Consider any $L\in\mathcal{L}$. It must meet at least one of the interiors of the vertical sides of at least one of $R_0$ or $R_2$, so that we can find $i\in\{0,2\}$ such that $P_i^{-1}(L)\in\mathcal{L}$. Starting from $D_1$, we can thus construct inductively a sequence $(D_{\omega_n})_{n>0}$ (with $\lvert\omega_n\rvert = n$) of elements of $\mathcal{L}$ whose slopes converge exponentially fast to $0$ from below.

Fix $\varepsilon=\varepsilon(p)>0$ small enough to ensure that all lines of slope in $[-\varepsilon,0]$ and $\varepsilon$-transversal to $R$ are $\varepsilon$-transversal to at least one of $R_0$ or $R_2$ (this is possible since $y_2<y_1$ and we consider lines of negative slope). Taking $n=n_0$ large enough, we can ensure the slope of $D_{\omega_{n_0}}$ is in $[-\varepsilon,0)$. If $D_{\omega_{n_0}}$ is $\varepsilon$-transverse to $R$, then we take $\omega_0=\omega_{n_0}$, $\ell_0=\lvert\omega_0\rvert$ and we are done: by assumption, we can inductively extend the word $\omega_0$ into words $\omega$ of arbitrary length such that $D_\omega$ is $\varepsilon$-transverse to $R$, i.e. $D$ is $\varepsilon$-transverse to $R_\omega$.

Otherwise, we continue applying $P_0^{-1}$ or $P_2^{-1}$; what is left to prove is that it is possible to ultimately produce a line $D_{\omega_{n_1}}$ that is $\varepsilon$-transverse to $R$. Up to reducing $\varepsilon$ further, we can assume that $P_0^{-1}([0,x_1]\times[0,2\varepsilon])\subset [0,1]\times[0,1-2\varepsilon]$ and $P_2^{-1}([x_2,1]\times[1-2\varepsilon,1])\subset [0,1]\times[2\varepsilon,1]$. Since its slope is between $-\varepsilon$ and $0$, for $D_{\omega_{n_0}}$ not to be $\varepsilon$-transverse to $R$ one of the two following conditions must hold: either $D_{\omega_{n_0}}$ meets the interior of the left side of $R$ at height less than $2\varepsilon$, or it meets the interior of the right side of $R$ at height more than $1-2\varepsilon$. We treat the first case, the second one being symmetrical.
We set $\omega_{n_0+k}=0^k\omega_{n_0}$, i.e. $D_{\omega_{n_0+k}}=P_0^{-k}(D_{\omega_{n_0}})$, where $k$ is the least integer such that $D_{\omega_{n_0+k}}$ meets the left side of $R$ at height at least $2\varepsilon$. For $\varepsilon$ small enough (precisely, $\varepsilon<y_1(2+2y_1)$), by minimality this height is between $2\varepsilon$ and $1-\varepsilon$, and $D_{\omega_{n_0+k}}$ is $\varepsilon$-transverse to $R$.
\end{proof}

\begin{proof}[Proof of Theorem \ref{thm:horseshoes} under assumption \ref{enumi:horse2}]
We fix $k\in\mathbb{N}$.  We shall find a word $\omega$ and a set $A$ of at least $k$ words, no one a  prefix of another, such that $J_{\omega\sigma}\supset I_{\omega}$ for all $\sigma\in A$. Then the $(I_{\omega\sigma})_{\sigma\in A}$ will have pairwise disjoint interiors, and $T(I_{\omega\sigma})=J_{\omega\sigma}\supset I_{\omega} \supset I_{\omega\tau}$ for all $\sigma,\tau\in A$.

Let $\ell_0\in\mathbb{N}, \varepsilon>0$ be as given by Lemma \ref{l:transverse}, then let $\eta>0$ and $n\in\mathbb{N}$ be as given by Proposition \ref{p:cover}. Let $\ell\ge \ell_0$ and $\omega$ be as given by Lemma \ref{l:transverse} with $\ell$ large enough that the slope of $D_\omega$ is less than $2\eta$ in absolute value. We consider the ``range'' of $D_\omega$, i.e. the interval $[a,b]$ where $a,b$ are the heights at which $D_\omega$ meets the right and left sides of $R$; by construction $b-a<2\eta$ and $\varepsilon \leq a \leq b<1-\varepsilon$. Let $y=\frac{a+b}{2}$, so that $y\in[\varepsilon,1-\varepsilon]$ and $[a,b]\subset[y-\eta,y+\eta]$. Finally, let
\[A= \#\big\{\sigma\in\{0,2\}^n \,\big|\, [y-\eta,y+\eta]\subset V_\sigma([0,1])\big\}.\]
Elements of $A$ are not prefix one to another, since they all have the same length. By Proposition \ref{p:cover}, there are at least $k$ elements in $A$ and $V_\omega^{-1}(I_\omega) \subset V_\sigma([0,1])$ for all $\sigma\in A$. By applying $V_{\omega}$, this means $J_{\omega\sigma}\supset I_\omega$, as desired.
\end{proof}

%

\subsection{Realisation of orders by orbits}\label{s:orders}

We now prove Corollaries \ref{cm:orders} and \ref{cm:universality}. We dismiss the notation of the beginning of the section, and prove the corollaries for any continuous map $T$ satisfying the conclusion of Theorem \ref{thm:horseshoes} (which is easily seen to be actually equivalent to the conclusion of Corollary \ref{cm:orders}).

\begin{proof}[Proof of Corollary \ref{cm:universality}]
Assume $\lvert\Omega\rvert=\{1,\dots,k\}$ and let $(I_1, \dots, I_k)$ be a horseshoe of $T$. For each periodic orbit $(i_1,\dots,i_p)$ of $S:\Omega\to\Omega$, we can find a $p$-periodic point $x$ of $T$ in $I_{i_1}$, whose orbits goes through $(I_{i_1},\dots,I_{i_p})$ in order ; we then define $\pi(i_j) = T^{j-1}(x)$ for all $j\in\{1,\dots, p\}$. All orbits of a finite dynamical system are ultimately periodic, so we only have left to define $\pi$ inductively on pre-periodic points, an $S$-antecedent $j$ of a point $i$ on which $\pi$ is already defined being sent to any $T$-antecedent of $T\circ \pi(i)$ in $I_j$.
\end{proof}

\begin{proof}[Proof of Corollary \ref{cm:orders}]
Assume $T:[0,1]\to[0,1]$ is a continuous map that admits horseshoes of arbitrary order. Let $k,\ell\in\mathbb{N}$ and consider 
any total strict order $\prec$ on the symbols $(s_i^j)_{1\le i\le k,0\le j\le \ell}$. We seek for points $x_1,\dots x_k\in [0,1]$ such that for all $i,j,i',j'$:
\[ s_i^j \prec s_{i'}^{j'} \Leftrightarrow T^jx_i < T^{j'}x_{i'}.\]

Given two closed subsets $I,J\subset [0,1]$, we write $I\le J$ whenever $\max I\le \min J$ and $I<J$ whenever $\max I<\min J$. Let $N$ be a positive integer that we shall made precise later on.
Let $I_1^0,I_2^0,\dots,I_{2N}^0$ be a horseshoe of $T$ numbered in increasing order, i.e. $I_i^0\le I_{i+1}^0$ for all $i\in\{1,\dots, 2N\}$. Then the family $(I_i=I_{2i}^0)_{1\le i\le N}$ is a horseshoe for $T$ with disjoint intervals (not only intervals of disjoint interiors). In particular $<$ induces a total strict order on $(I_i)_{1\le i\le N}$.

For all word $\omega=\alpha_0\alpha_1\dots \alpha_n\in \{1,\dots,N\}^*$, we let $I_\omega = I_{\alpha_0} \cap T^{-1}(I_{\alpha_1}) \cap \dots\cap T^{-n}(I_{\alpha_n})$, i.e. $I_\omega$ is the set of points whose orbit under $T$ runs over the $I_i$ as specified by $\omega$. We know each $I_\omega$ is a non-empty compact set, we can thus choose a point $y_\omega\in I_\omega$ for each word $\omega$ of length $\ell$, and by construction $T^j(y_{\alpha_0\alpha_1\dots \alpha_\ell}) < T^{j'}(y_{\beta_0\beta_1,\dots \beta_\ell}) \Leftrightarrow I_{\alpha_j} < I_{\beta_{j'}}$.

Assuming $N\ge 3k+2$, we can find a map $\alpha_0:\{1,\dots,k\} \to \{1,\dots,N\}$ such that
\begin{itemize}
\item $I_{\alpha_0(i)}<I_{\alpha_0(i')} \Leftrightarrow s_i^0 \prec s_{i'}^0$, and
\item there is a distance at least $N_0 = \frac{N-k}{k+1}-1 \ge 1$ between any two numbers in the family $(0,\alpha_0(1),\dots,\alpha_0(k),N)$
\end{itemize}
i.e. $\alpha_0$ selects intervals in the horseshoe that reflect the order on the $s_i^0$ and are almost evenly spread inside $\{1,\dots, N\}$.

Assuming $N$ is large enough, we can in the same way construct inductively on $j$ from $0$ to $\ell$ a sequence of positive integers $N_j$ and maps $\alpha_j:\{1,\dots,k\} \to \{1,\dots,N\}$ such that
\begin{itemize}
\item $I_{\alpha_j(i)}<I_{\alpha_{j'}(i')} \Leftrightarrow s_i^j \prec s_{i'}^j$, and
\item there is a distance at least $N_j = \frac{N_{j-1}-k}{(j+1)k+1}-1 \ge 1$ between any two numbers in the family containing $0$, $N$ and all the $\alpha_{j'}(i)$ for $0\le j'\le j$ and $1\le i\le k$.
\end{itemize}

Then the points $x_i=y_{\alpha_0(i)\alpha_1(i)\dots\alpha_\ell(i)}$ have the desired property and the proof of Corollary \ref{cm:orders} is complete.
\end{proof}


\section{Relative metric mean dimension and other high-complexity measurements}\label{s:meanDim}

Let us recall the metric mean dimension introduced by \cite{Lindenstrauss-Weiss2000mean}. If $T:\Omega\to\Omega$ is a dynamical system on a compact space $\Omega$ endowed with a metric $d$, one defines the Bowen metrics by
\[d_n(x,y) = \max\{d(T^k(x),T^k(y)) \mid 0\le k < n \} \qquad \forall n\in\mathbb{N}, \forall x,y\in\Omega.\]

A set $S\subset \Omega$ is said to be $(n,\varepsilon)$-separated when $d_n(x,y)\ge\varepsilon$ for all $x\neq y\in S$. Denoting by 
$N(d,T,\varepsilon,n)$ the cardinal of a largest $(n,\varepsilon)$-separated set, one defines the \emph{metric mean dimension relative to $d$} by
\[\mdim(T,d) := \liminf_{\varepsilon\to 0} \limsup_{n\to\infty} \frac{\log N(d,T,\varepsilon,n)}{n \log\frac{1}{\varepsilon}}.\]
In other words, a map with metric mean dimension (relative to $d$) $\delta$ has $N(d,T,\varepsilon,n)$ roughly of the order $(1/\varepsilon^\delta)^n$ for small fixed $\varepsilon$ as $n\to\infty$, or ``entropy of the order of $1/\varepsilon^\delta$ at scale $\varepsilon$.''

This is not a topological invariant, it depends on the choice of metric. For example, for every $\alpha\in(0,1)$ the function $d^\alpha$ is still a metric, and the mean dimension relative to $d^\alpha$ is $1/\alpha$ times the mean dimension relative to $d$: as soon as $T$ has positive mean dimension relative to some metric, one can find metrics for which it has arbitrarily high metric mean dimension. Lindenstrauss and Weiss introduced the topological invariant \emph{(absolute) metric mean dimension}
\[\mdim(T) := \inf_{d} \mdim(T,d)\]
where $d$ runs over all metrics on $\Omega$ inducing its topology. We shall use the adjective ``absolute'' when there is a risk of confusion with the metric mean dimension relative to a particular metric, but usually $\mdim(T)$ will simply be called the metric mean dimension of $T$.

The main goal of this quantity was to bound from above another topological invariant of high-complexity maps, the (topological) mean dimension, introduced by Gromov. We shall not give the definition here, as this quantity is of little interest in our case: as soon as $\Omega$ has finite topological dimension, the mean dimension of every continuous map on $\Omega$ vanishes. Examples of positive mean dimension system include the shifts on $([0,1]^k)^{\mathbb{N}}$, which have mean dimension $k$ by a theorem of Lindenstrauss and Weiss (and are thus not topologically conjugated one to another for different $k$).

A prominent question is whether metric mean dimension is always equal to mean dimension (a positive answer is conjectured; see e.g. \cite{lindenstrauss-Tsukamoto2019double}). As mentioned above, it is known from \cite{Lindenstrauss-Weiss2000mean} that the metric mean dimension is never lower than the mean dimension, and no example is known where they differ. While during the course of this research we had hoped that $\mdim(Z_p)>0$ for some $p$, we will on the contrary prove in Section \ref{s:vanishing} that $\mdim(T)=0$ for all interval maps, proving a small case of the conjecture.

However, zipper maps and similar examples deserve an invariant that quantifies their level of chaos beyond infinite entropy. The purpose of this section is to propose and study several quantities playing this role and bound them in the case of zipper (or more general) maps.

\subsection{Metric mean dimension relative to the Euclidean metric}

We first bound from below the metric mean dimension relative to the Euclidean metric $\lvert\cdot\rvert$ for the class of hypersensitive maps, with Theorem \ref{t:mdim} as a consequence in view of Proposition \ref{p:hypersensitivity}.
\begin{theo}\label{t:mdim-general}
If $T:[0,1]\to [0,1]$ is a $\beta$-hypersensitive interval map, then 
\[\mdim(T,\lvert\cdot\rvert)\ge 1-\beta.\]
\end{theo}

\begin{proof}
By assumption there is some $C>0$ such that $\lvert T(A) \rvert \ge C\lvert A\rvert^\beta$ for all interval $A\subset [0,1]$. Let $\varepsilon\in(0,\frac14)$ and consider any partition $[0,1]=I_1\cup I_2 \dots I_k$ into $k=\lfloor\frac1{2\varepsilon}\rfloor$ intervals, each of length at least $2\varepsilon$ and at most $2\varepsilon+8\varepsilon^2$. Let $I'_j$ be the centered subinterval of $I_j$ of length $\varepsilon$. For each $j\in\{1,\dots, k\}$ it holds $\lvert T(I_j)\rvert \ge C\varepsilon^{\beta}$, and thus $T(I_j)$ contains at least $C'\varepsilon^{\beta-1}$ intervals $I_\ell$ of the partition for some $C'$ not depending on $\varepsilon$.

Construct a directed graph $G$ whose set of vertices is $\{1,2,\dots,k\}$, and where there is an edge from $j$ to $\ell$ whenever $T(I'_j)\supset I_\ell$. By the intermediate value theorem, for any directed path $j_0 \to j_1 \to\dots\to j_{n-1}$ in $G$ there is some starting point $x\in\Omega$ such that $T^i(x)\in I'_{j_i}$ for all $i\in\llbracket 0,n-1\rrbracket$. By construction, any two different paths corresponds to $(n,\varepsilon)$ separated starting points. Since $G$ has minimal out-degree at least $C'\varepsilon^{\beta-1}$, for any $\beta'>\beta$ and any small enough $\varepsilon$ the graph $G$ has at least $\varepsilon^{n(\beta'-1)}$ paths of length $n$. It follows that $\mdim(T) \ge 1-\beta'$, and the desired conclusion follows from letting $\beta'$ approach $\beta$.
\end{proof}

\subsection{Mean dimension for measured dynamical systems}

In this section, we develop an analogue of the Kolmogorov-Sinai entropy for mean dimension. Endow the (metric, compact) phase space $\Omega$ with its Borel $\sigma$-algebra, let $T:\Omega\to \Omega$ be a measurable map and assume $\mu$ is a $T$-invariant probability measure on $\Omega$.

\subsubsection{Kolmogorov-Sinai mean dimension}

Let $\xi=(B_1,\dots,B_k)$ be a (measurable) partition of $\Omega$, and recall that one defines
\[H(\mu;\xi) := -\sum_{i=1}^k \mu(B_i) \log \mu(B_i).\]
The entropy of $(T,\mu)$ relative to the partition $\xi$ is 
\[h(T,\mu,\xi) = \lim_{n\to\infty} \frac{H\big(\mu; \bigwedge_{j=0}^{n-1} T^{-j}(\xi) \big)}{n}\]
and the Kolmogorov-Sinai entropy of $(T,\mu)$ is then defined as
$h(T,\mu) = \sup_{\xi} h(T,\mu,\xi)$
where $\xi$ runs over all partitions into finitely many subsets.

In a spirit similar to the metric mean dimension, let us modify this definition to make it meaningful in the case of infinite entropy, by looking how the asymptotic complexity of the dynamically refined partitions $\bigwedge_{j=0}^{n-1} T^{-j}(\xi)$ increase with the number of elements in $\xi$. Let $\ptt{k}$ be the set of measurable partitions into at most $k$ subsets. We define the \emph{Kolmogorov-Sinai mean dimension} of $(T,\mu)$ by
\[\mdKS(T,\mu) := \liminf_{k\to \infty} \sup_{\xi \in\ptt{k}} \frac{h(T,\mu,\xi)}{\log k}  \]
Observe that if $\xi\in\ptt{k}$, then $\bigwedge_{j=0}^{n-1} T^{-j}(\xi)$ has at most $k^n$ elements, so that 
\[H\big(\mu; \bigwedge_{j=0}^{n-1} T^{-j}(\xi) \big) \le n\log k.\]
The Kolmogorov-Sinai mean dimension is thus bounded above by $1$, and we shall see that the relation with the metric mean dimension will involve some dimension of $\Omega$. The Kolmogorov-Sinai mean dimension is more a proportion (of the maximal possible complexity) than a dimension, but we still chose this name in order to draw a parallel between entropies and mean dimensions.

If $\lambda$ is the Lebesgue measure on $[0,1]$, then the shift on $[0,1]^{\mathbb{N}}$ preserves $\lambda^{\otimes \mathbb{N}}$ and it is easy to see that in this case the Kolmogorov-Sinai mean dimension is $1$. Any measurable dynamical system with finite Kolmogorov-Sinai entropy has vanishing Kolmogorov-Sinai dimension. We will not compute any Kolmogorov-sinai mean dimension here, but will prove lower bounds. It is an intriguing problem to find explicit examples $(T,\mu)$ where one can compute $\mdKS(T,\mu)$ and with $0<\mdKS(T,\mu)<1$.

\subsubsection{Half a variational principle}

In this section, we shall prove that on finite-dimensional phase spaces, the metric mean dimension can be used to build invariant measure of positive Kolmogorov-sinai mean dimension (and Corollary \ref{c:KS} will follow immediately).

We shall use  the \emph{upper Minkowski dimension} for metric spaces:
\[\udimM(\Omega,d) = \limsup_{\varepsilon\to 0} \frac{\log N(\Omega,d,\varepsilon)}{\log\frac1\varepsilon}\]
where $N(\Omega,d,\varepsilon)$ is the maximal cardinal of an $\varepsilon$-separated subset of $\Omega$.

\begin{theo}\label{t:variational}
For all continuous map $T:\Omega\to\Omega$ and all metric $d$ on $\Omega$ such that $\udimM(\Omega,d)\in(0,\infty)$,
\[\sup_{\mu\in\prob_T(\Omega)}\mdKS(T,\mu) \ge \frac{\mdim(T,d)}{\udimM(\Omega,d)} \]
\end{theo}

\begin{proof}
We follow a standard construction \cite{Misiurewicz1977short,Katok-Hasselblatt1995introduction}. Set $D=\udimM(\Omega,d)$ and fix any $\eta>0$. For each $k\in\mathbb{N}$, set $\varepsilon(k)=(\frac1k)^{\frac{1}{D+\eta}}$. By definition of the upper Minkowski dimension, whenever $k$ is large enough, $N(\Omega,d,\varepsilon(k)/2)\le k$ (the factor $1/2$ is harmless and will be used later on). Consider from now on a fixed, large enough $k\in\mathbb{N}$ and let $\varepsilon=\varepsilon(k)$.

Let $n\in\mathbb{N}$, consider $E_n$ a maximal $(n,\varepsilon)$-separated set with respect to $T$, and define
\[ \nu_n = \frac{1}{\lvert E_n\rvert} \sum_{x\in E_n} \delta_n, \qquad 
  \mu_n = \frac{1}{n} \sum_{i=0}^{n-1} T^i_* \nu_n \]
Let $\mu$ be a cluster point of the sequence $(\mu_n)$; since $T$ is continuous, $\mu$ is $T$-invariant. Let $\{x_1,\dots x_{k'}\}$ be a maximal $\varepsilon/2$-separated set; by maximality it is an $\frac\varepsilon2$ covering and
we can then construct as in Lemmas 4.5.1 and 4.5.2 of \cite{Katok-Hasselblatt1995introduction} a partition $\xi$ of $\Omega$ into $k'$ elements such that $\limsup_n \frac{\log \lvert E_n\rvert}{n} \le h_\mu(T,\xi)$. We get
\begin{align*}
\sup_{\xi\in\mathcal{P}_k}\frac{h_\mu(T,\xi)}{\log k}  &\ge \limsup_{n\to\infty} \frac{\log N(d,T,\varepsilon(k),n)}{n(D+\eta)\log\frac1{\varepsilon(k)}} \\
\mdKS(T,\mu) &\ge \liminf_{k\to\infty} \limsup_{n\to\infty} \frac{\log N(d,T,\varepsilon(k),n)}{n(D+\eta)\log\frac1{\varepsilon(k)}}\\
  &\ge \liminf_{\varepsilon\to0}\limsup_{n\to\infty} \frac{\log N(d,T,\varepsilon,n)}{n(D+\eta)\log\frac1{\varepsilon}}\\
    &\ge \frac{\mdim(T,d)}{D+\eta}\\
\sup_{\mu\in\prob_T(\Omega)}\mdKS(T,\mu) &\ge \frac{\mdim(T,d)}{D}.
\end{align*}
\end{proof}

We cannot hope for equality in Theorem \ref{t:variational} since in Section \ref{s:vanishing} we will construct for each continuous $T:[0,1]\to[0,1]$ a metric $d$ on the interval such that $\udimM([0,1],d)=1$ and $\mdim(T,d)=0$. However, the possibility of a full variational principle is briefly discussed in the next section.

\subsubsection{Other topological invariants}

Let us consider other topological invariants that could be used to distinguish highly irregular dynamical systems on finite-dimensional spaces, circumventing the vanishing of absolute mean dimension for interval maps.

First, recall that whenever $d$ is a metric on the phase space $\Omega$ and $\alpha\in(0,1)$, $d^\alpha$ also is a metric (this operation is sometimes called \emph{snowflaking}), with $\udimM(\Omega,d^\alpha) = \alpha^{-1}\udimM(\Omega,d)$ and $\mdim(T,d^\alpha)=\alpha^{-1}\mdim(T,d)$. 
A first possibility would thus be to consider
\[\sup_d \frac{\mdim(T,d)}{\udimM(\Omega,d)}\]
with the small inconvenient that, since $N(d,T,\varepsilon,n)\le N(\Omega,d,\varepsilon)^n$, we get a quantity that is bounded above by $1$ whenever defined, as in the Kolmogorov-Sinai mean dimension. The same default appears with the more natural version where $\udimM(\Omega,d)$ is replaced by the \emph{lower} Minkowski dimension.

A second possibility is to take a supremum over a restricted set of ``minimal'' distances. For example,
\[\sup \big\{\mdim(T,d) \colon \udimM(\Omega,d)=\dim \Omega\big\}\]
where $\dim$ is the topological dimension.
A third possibility is to prevent snowflaking by asking geometric condition on the metrics, for example by considering:
\[\sup \big\{\mdim(T,d) \colon (\omega,d) \text{ is a length space}\big\}.\]
The last two quantities are suitable when $\Omega$ is a manifold, but not when it is a Cantor space (then its topological dimension is $0$ and it has no length metric).

All three of these quantities are topological invariants, and positive for hypersensitive maps of the interval by Theorem \ref{t:mdim-general}; but we will not name them here, as it would take more examples to decide which ones are ultimately relevant; two clues that would help make this decision are:
\begin{itemize}
\item to use one of these invariants to prove that two explicit dynamical systems are not topologically conjugated;
\item to obtain a full variational principle, i.e. an equality between $\sup_\mu \mdKS(T,\mu)$ and one of these quantities (possibly normalized by a metric dimension).
\end{itemize}

These problems are beyond the scope of the current article, and we leave them open to further investigations.

\section{Vanishing of the absolute metric mean dimension for interval maps}\label{s:vanishing}

In this section we prove Theorem \ref{t:vanishing}. Let $T:[0,1]\to[0,1]$ be a continuous map. We will construct the metric $d$ by conjugating the Euclidean metric, i.e. 
\[d(x,y)= \lvert h(x)-h(y)\rvert, \qquad \forall x,y\in[0,1],\]
where $h:[0,1]\to[0,1]$ is a homeomorphism to be constructed. We actually design $h$ with little reference to $T$, only using its modulus of continuity; it will compress large regions of the interval and expand tiny regions, making all specific features of $T$ only appear at very small scales in $h\circ T\circ h^{-1}$.

\subsection{A family of homeomorphism of the interval}

We construct a family of homeomorphisms of $[0,1]$, among which $h$ shall be chosen, in a way similar to the construction of zipper maps. Here we only use increasing changes of coordinates, so as to send homeomorphisms to homeomorphisms, and we use $4$ pieces to intertwine two regions of high slope with two region with small slopes. For each $\alpha\in(0,\frac12)$ (to be thought of as very small), we define a map $\Psi^\alpha:\Ck^0_0 \to \Ck^0_0$ by
\[\Psi^\alpha f(x) = \begin{dcases*}\frac{1-\alpha}{2}  f\big(\frac{2}{\alpha}x\big) & if $x\in[0,\frac\alpha2]$ \\               
      \frac{1-\alpha}{2} +\frac\alpha2 f\big(\frac{2}{1-\alpha}x-\frac{\alpha}{1-\alpha}\big) & if $x\in[\frac\alpha2,\frac12]$ \\
      \frac12+ \frac{1-\alpha}{2}f\big(\frac2\alpha x -\frac1\alpha\big) & if $x\in[\frac12,\frac{1+\alpha}{2}]$ \\
      1-\frac{\alpha}{2}+\frac\alpha2 f\big(\frac{2}{1-\alpha}x-\frac{1+\alpha}{1-\alpha}\big) & if $x\in[\frac{1+\alpha}{2},1]$
              \end{dcases*} \] 
The graph of $\Psi^\alpha(\Id)$ is shown in Figure \ref{f:homeo}. As in the definition of zipper maps, it will be convenient to rephrase this definition through a two-dimensional iterated function system acting on graphs of functions: for all $\alpha\in(0,\frac12)$, all $x,y\in[0,1]^2$ we set
\begin{align*}
H^\alpha_0(x) &= \frac\alpha2 x  & V^\alpha_0(y) &= \frac{1-\alpha}{2}y & P^\alpha_i(x,y) &= (H^\alpha_i(x),V^\alpha_i(y)) \\
H^\alpha_1(x) &=\frac{1-\alpha}{2}x+\frac{\alpha}{2} & V^\alpha_1(y) &=\frac{\alpha}{2}y+\frac{1-\alpha}{2} && \forall i\in\{0,1,2,3\}\\
H^\alpha_2(x) &=\frac\alpha2 x+\frac{1}{2} & V^\alpha_2(y) &=\frac{1-\alpha}{2}y+\frac{1}{2}\\
H^\alpha_3(x) &=\frac{1-\alpha}{2}x+\frac{1+\alpha}{2} & V^\alpha_3(y) &=\frac{\alpha}{2}y+1-\frac{\alpha}{2}
\end{align*}
For any $f\in\Ck_0^0$ of graph $G_f$, $\Psi^\alpha f$ has graph $\cup_i P_i(G_f)$ and sends each interval $I_i := H_i([0,1])$ onto $J_i := V_i([0,1])$.

\begin{figure}[htp]
\centering
\includegraphics[width=.4\linewidth]{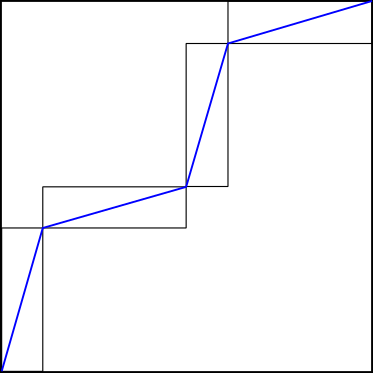}
\caption{The image of the identity map under $\Psi_\alpha$ (in blue). Each box has length $\frac{1-\alpha}{2}$ and width $\frac\alpha2$, so that the slopes of the blue segments are $\frac{1-\alpha}{\alpha}$ and $\frac{\alpha}{1-\alpha}$.}
\label{f:homeo}
\end{figure}

Consider a sequence $s=(s_n)\in (0,\frac12)^{\mathbb{N}}$. For all $n\in\mathbb{N}$, $\Psi^{s_1}\circ\dots\circ\Psi^{s_n}$ is a contraction on $\Ck^0_0$ of ratio bounded by $ \big(\frac{1-\alpha}{2}\big)^n$ and sends increasing homeomorphisms to increasing homeomorphisms. The sequence $\big(h^n=\Psi^{s_1}\circ\dots\circ\Psi^{s_n}(\Id)\big)_n$ is thus a Cauchy sequence and converges to a non-decreasing map $h=h_s\in\Ck^0_0$.

\begin{lemm} \label{lemm:h}
For all $s=(s_n)\in (0,\frac12)^{\mathbb{N}}$, the map $h=h_s$ is a homeomorphism of $[0,1]$ to itself.
\end{lemm}

We introduce some more notation that will be used here and later on. Letting the dependency on $s$ implicit, for each finite word $\omega=i_1\dots i_\ell\in\{0,1,2,3\}^*$ we consider the intervals
\[I_\omega=H^{s_1}_{i_1} \circ \dots \circ H^{s_\ell}_{i_\ell}([0,1]) \text{ and }
J_\omega=V^{s_1}_{i_1} \circ \dots \circ V^{s_\ell}_{i_\ell}([0,1]).\]

\begin{proof}
We already know that $h$ is in $\Ck_0^0$ and is non-decreasing, we have left to prove it is increasing. By construction, each $h^n$ (and thus $h$) sends $I_\omega$ to $J_\omega$ for each $\omega=i_1\dots i_\ell\in\{0,1,2,3\}^*$.
Observe that $\lvert I_\omega\rvert \le \frac{1}{2^\ell}$ (this is the motivation to subdivide the interval in no less than $4$ parts), thus goes uniformly to zero as $\ell$ goes to infinity. For any $x_1<x_2\in[0,1]$, there exist a finite word $\omega$ such that $I_\omega\subset (x_1,x_2)$; then $J_\omega\subset [h(x_1),h(x_2)]$ and since $\lvert J_\omega\rvert>0$, we obtain $h(x_2)>h(x_1)$.
\end{proof}

\def\I{\mathcal{I}}
\def\L{\mathcal{L}}
\def\ve{\varepsilon}
\def\es{\varnothing}

\subsection{Choosing the parameters $s$}

We consider $s=(s_n)_{n\ge 1}$ a \emph{decreasing} sequence and the corresponding homeomorphism $h$; we will add assumptions on $s$ at several points below, but we prefer to introduce them only when needed. We denote by $d$ the corresponding metric on $[0,1]$, given by $d(x,y)=\lvert h(x)-h(y)\rvert$.

We consider the rectangles $R_\omega=I_\omega\times J_\omega$. The \emph{depth} of $I_\omega$, $J_\omega$ and $R_\omega$ is simply the length $\lvert\omega\rvert$ of their defining word. For any positive integer $p$, the rectangles of depth $p$ cover the graph of $h$ and form a \emph{chain} $R_{0\dots 00}, R_{0\dots01}, \dots R_{3\dots 32}, R_{3\dots 33}$, with successive rectangles touching only at vertices and non-sucessive rectangles disjoint.

A rectangle of depth $p$ have area $4^{-p} \prod_{j=1}^p s_j(1-s_j)$, independent of the specifics of the word $\omega=i_1\dots i_p$. Let $j(\omega)\subset\llbracket1,p \rrbracket$ be the set of indices $j$ such that $i_j$ is even; then
\begin{align*}
\lvert I_\omega \rvert  &= 2^{-p} \prod_{j\in j(\omega)} s_j \prod_{\substack{j\notin j(\omega)\\1\le j\le p}} (1-s_j) \, , \\
\lvert J_\omega \rvert  &= 2^{-p} \prod_{j\in j(\omega)} (1-s_j) \prod_{\substack{j\notin j(\omega)\\1\le j\le p}} s_j \, .
\end{align*}
We will assume that $(s_n)$ decreases very fast, more precisely
\begin{align}
\prod_{i\ge 1} (1-s_i) &> 1-2^{-10} \label{e:s1} \, ,\\
\forall n\ge 1,\qquad\qquad s_n &< 2^{-n-10}\prod_{j=1}^{n-1} s_j =: P_n \, .
\label{e:s2}
\end{align}
In particular, the aspect ratio of any rectangle $R_\omega$ of depth $p$ is mostly determined by the last letter of $\omega$ alone; when that letter is even,
\[
\lvert I_\omega\rvert <2^{-p} s_p <2^{-2p-10}\prod_{j=1}^{p-1} s_j \le 2^{-2p-9} (1-s_p) \prod_{j=1}^{p-1} s_j \le 2^{-p-9} \lvert J_\omega\rvert \, ,
\]
and $R_\omega$ is called a \emph{vertical} rectangle; similarly when that letter is odd, 
\[\lvert I_\omega\rvert > 2^{-p} \Big(\prod_{j=1}^{p-1} s_j\Big) \frac12  > 2^9 s_p > 2^{p+9} \lvert J_\omega\rvert\]
and $R_\omega$ is called a \emph{horizontal} rectangle. For all $p$, the chain of rectangles of depth $p$ alternates between vertical and horizontal rectangles.

Let $\varepsilon\in(0,1)$; for rounding issues, we set $\varepsilon'=1/(\lfloor1/\varepsilon\rfloor+1)$ so that $\varepsilon'$ is the inverse of an integer with $\frac\varepsilon2 \le \varepsilon' < \varepsilon$. We consider the following cover of $[0,1]$ by sets of $d$-diameter bounded by $\varepsilon$:
\[\mathcal{I}_\varepsilon := \big\{ h^{-1}([k\varepsilon',(k+1)\varepsilon'])\colon k\in\llbracket 0,1/\varepsilon'-1 \rrbracket \big\} \,.\]
We group them by size, denoting for all $k\ge 1$:
\[\mathcal{I}_\varepsilon(k) := \big\{ I\in\mathcal{I}_\varepsilon \,\big|\, s_k\le \lvert I\rvert <s_{k-1} \big\} \quad\text{and}\quad \mathcal{L}_\varepsilon(k) = \bigcup_{j=1}^k \mathcal{I}_\varepsilon(k) \,,\]
where by convention $s_0=1$ and $\mathcal{L}$ stands for ``large'': $\mathcal{L}_\varepsilon(k)$ is the set of intervals in $\mathcal{I}_\varepsilon$ whose length is at least $s_k$.

Finally, $p_0=p_0(\varepsilon)$ denotes the integer such that $\varepsilon'\in [s_{p_0},s_{p_0-1})$.

\begin{lemm}\label{l:horizontal}
Let $m\in\mathbb{N}$, consider an interval $I\subset [0,1]$ and set $J=h(I)$.
\begin{enumerate}
\item If $\lvert I\rvert \ge s_m$, then there exist a horizontal rectangle $R_\omega$ of depth $m$ such that $\lvert I\cap I_\omega\rvert >0$, and also $\lvert J\cap J_\omega\rvert >0$,
\item if $\lvert I\rvert < s_m$, then there exist at most $3$ rectangles $R_\omega$ of depth $m$ such that $I\cap I_\omega\neq \varnothing$.
\item if $\lvert J\rvert \ge s_m$, then there exist a vertical rectangle $R_\omega$ of depth $m$ such that $\lvert J\cap J_\omega\rvert >0$, and also $\lvert I\cap I_\omega\rvert >0$.
\end{enumerate}
\end{lemm}

\begin{proof}
Assume $\lvert I\rvert \ge s_m$. Then there are no vertical rectangle $R_\omega$ of depth $m$ such that $I\subset I_\omega$, since they have width $\lvert I_\omega\rvert < s_m$. Since the chain of depth $m$ rectangle alternates between horizontal and vertical ones, there must exist a horizontal rectangle $R_\omega$ of depth $m$ such that $\lvert I\cap I_\omega\rvert >0$. Since $J=h(I)$ and $J_\omega=h(I_\omega)$, we also have $\lvert J\cap J_\omega\rvert >0$.

Assume $\lvert I\rvert < s_m$, then horizontal rectangles $R_\omega$ of depth $m$ have width $\lvert I_\omega\rvert > 2^9 s_m > \lvert I\rvert$, so that among rectangles of depth $m$, at most two successive horizontal rectangle and the vertical rectangle between them can satisfy $I\cap I_\omega\neq\varnothing$.

The last item is proved as the first one, inverting coordinate axes.
\end{proof}

\begin{lemm} \label{lemm:cardL}
For all $k\ge 1$ and all $\varepsilon\in(0,1)$, 
\[\card(\mathcal{L}_\varepsilon(k)) \le 4^{\max(k,p_0)}\]
\end{lemm}

\begin{proof}
Write $m=\max(k,p_0)$. Let $I\in\mathcal{L}_\varepsilon(k)$ and set $J=h(I)$; then $\lvert I\rvert\ge s_k\ge s_m$ and $\lvert J\rvert=\varepsilon'\ge s_{p_0}\ge s_m$. Applying Lemma \ref{l:horizontal}, there exist a horizontal rectangle $R_\omega$ of depth $m$ such that $\lvert J_\omega\cap J\rvert >0$. Such a rectangle has height $\lvert J_\omega\rvert<s_m\le \lvert J\rvert$, and $J_\omega$ can therefore meet at most two of the possible $J$ when $I$ runs over $\mathcal{L}_\varepsilon(k)$, since these $J$ have disjoint interior. We obtain the desired conclusion by counting that there are $\frac12 4^m$ horizontal rectangle of depth $m$.
\end{proof}

Now we add the last assumption to $s$, which will depend on the modulus of continuity of the considered map $T:[0,1]\to[0,1]$.
We take $s_1$ small enough to ensure that \eqref{e:s2} implies \eqref{e:s1}, and additionally, inductively on $k\in\mathbb{N}$ we take $s_k$ small enough to ensure both \eqref{e:s2} and
\begin{equation}
\sup\big\{\lvert T^j(I)\rvert \colon 0\le j\le k, I\subset [0,1] \text{ an interval with }\lvert I\rvert \le s_k\big\} <P_k.\label{e:s3}
\end{equation}

\subsection{Relating intervals of the cover and rectangles}

Given a collection $\mathcal{I}$ of intervals and an interval $I_0$, we set
\[I_0\pitchfork \mathcal{I} = \{I\in\mathcal{I} \mid I\cap I_0\neq\varnothing\} \,.\]

\begin{lemm} \label{lemm:orbitBound}
Let $k\ge p_0$. For any interval $I_0\subset[0,1]$ such that $\lvert I_0\rvert<s_{k-2}$ and $q \geq 0$,
\[ \card\big(I_0\pitchfork \mathcal{L}_\varepsilon(k+q)\big)\le 4^{q+3} \,. \]
\end{lemm}

\begin{proof}
Since $\lvert I_0\rvert<s_{k-2}$, by Lemma \ref{l:horizontal} there are at most $3$ rectangles $R_\omega$ of depth $k-2$ such that $I_0\cap I_\omega\neq\varnothing$. Since each rectangle contains exactly $4$ rectangles of one more depth, there are at most $3\cdot4^{q+2}$ rectangles $R_\omega$ of depth $k+q$ such that $I_0\cap I_\omega\neq\varnothing$, among which at most $6 \cdot 4^{q+1}$ are vertical.

Let $I\in I_0\pitchfork \mathcal{L}_\varepsilon(k+q)$. Since $\lvert h(I)\rvert =\varepsilon'\ge s_{p_0}\ge s_{k+q}$, Lemma \ref{l:horizontal} ensures that there exists a vertical rectangle $R_\omega$ of depth $k+q$ such that $\lvert I\cap I_\omega\rvert>0$. Since $\lvert I_\omega\rvert < s_{k+q} \le \lvert I\rvert$, any such $I_\omega$ can meet at most $2$ of the elements of $\mathcal{L}_\varepsilon(k+q)$.

On one hand, since $\I_\varepsilon(k+q)$ is a tiling, at most two of the elements of $I_0\pitchfork \mathcal{L}_\varepsilon(k+q)$ are not contained in $I_0$, one containing each endpoint of $I_0$. On the other hand, given an element of $I_0\pitchfork \mathcal{L}_\varepsilon(k+q)$ contained in $I_0$ and a vertical rectangle $R_\omega$ of depth $k+q$ such that $\lvert I\cap I_\omega\rvert>0$, we have $I_0\cap I_\omega\neq\varnothing$. We thus have at most $3\cdot 4^{q+2}+2$ elements in $I_0\pitchfork \mathcal{L}_\varepsilon(k+q)$.
\end{proof}

\begin{coro} \label{coro:orbitBound}
For all $k\ge p_0$, $q \ge 0$, $j\in\llbracket 0,k\rrbracket$ and all $I\in \I_\varepsilon(k)$, 
\[\card\big(T^j(I)\pitchfork \mathcal{L}_\varepsilon(k+q)\big)\le 4^{q+3} \,. \]
\end{coro}
This means that when coding the orbits of $T$ by the elements of $\mathcal{I}_\varepsilon$ they visits, visiting a small interval (in some $\mathcal{I}_\varepsilon(k)$ with large $k$) means there are only a limited number of ways to continue the orbit while not visiting even much smaller intervals. 

We now bound the maximal depth of intervals of $\mathcal{I}_\varepsilon$.
\begin{lemm} \label{lemm:depth}
Let $K_\ve:= \lceil \log_2 \frac1\ve \rceil+2$.
Then $\mathcal{I}_\varepsilon=\mathcal{L}_\varepsilon(K_\varepsilon)$, i.e. for all $k>K_\varepsilon$, $\mathcal{I}_\varepsilon(k)=\varnothing$.
\end{lemm}
\begin{proof}
Let $I\in \mathcal{I}_\varepsilon$. For every $\omega$ of depth $K_\ve$,
\[ \lvert J_\omega\rvert < 2^{-K_\varepsilon} \le \frac{\ve}{4} \le \frac{\varepsilon'}{2} = \frac12 \lvert h(I) \rvert \,,\]
and if $R_\omega$ is horizontal, $\lvert J_\omega\rvert \ll \frac12 \lvert h(I) \rvert$.
There must thus be an horizontal rectangle $R_\omega$ such that
\[
J_{\omega} \subset h(I) \,.
\]
Indeed, consider $\omega_1$ such that $J_{\omega_1}$ contains the minimum of $h(I)$; if $R_{\omega_1}$ is vertical, take $\omega$ the next word, otherwise the subsequent one.

Now $|I_{\omega}| > s_{K_\ve}$ and $I_{\omega'} \subset I$, therefore $I \in \L_\ve(K_\ve)$.
\end{proof}

\subsection{End of the proof of Theorem \ref{t:vanishing}      }

Let $T:[0,1] \to [0,1]$ be continuous. There exists a decreasing sequence $s$ satisfying conditions \eqref{e:s1}, \eqref{e:s2} and \eqref{e:s3}. We consider the homeomorphism $h:=h_s$ of $[0,1]$ given by Lemma \ref{lemm:h}. We use the metric $d:=d_h$, given for all $x,y \in [0,1]$ by
$d(x,y)=|h(x)-h(y)|$, and we have to show that 
\[
\lim_{\ve\to 0} \lim_{n\to\infty} \frac{\log N(d,T,\varepsilon,n)}{n \log \frac1\ve} = 0 \,,
\]
where $N(d,T,\varepsilon,n)$ is the largest possible cardinal of a $(n,\varepsilon)$-separated set.

Let us fix $0 < \ve \ll 1$ for now, and as above set $\varepsilon'=1/(\lfloor \frac1\varepsilon\rfloor +1)$  and $p_0=p_0(\ve)$ such that $\ve' \in [s_{p_0}, s_{p_0-1})$. We may assume that $p_0>3$ by taking $\varepsilon$ small enough. Condition \eqref{e:s2} and $s_j<\frac12$ ensure the (very conservative) estimate 
\begin{equation} \label{equ:n2}
\log_2 \frac1\varepsilon \ge \log_2\frac{1}{\varepsilon'} -1 \ge (p_0+9)(p_0-2)-1 > p_0^2 \,. 
\end{equation}

Define a vertex-labeled oriented graph $G=(V,E,\delta)$ by the vertex set $V=\mathcal{I}_\varepsilon$, the edge set $E$ of all $(I,I')$ such that $T(I)\cap I'\neq \varnothing$, and the labeling map $\ell:V\to \mathbb{N}$ such that $I\in \mathcal{I}_\varepsilon(\ell(I))$ for all $I\in\mathcal{I}_\varepsilon$. For each $n\in\mathbb{N}$, let $G^n$ be the set of oriented paths in $G$ of length $n$. It provides an $(\varepsilon,n)$-cover of $[0,1]$ for the metric $d$: by associating to a path $\alpha=(I_0,\dots,I_{n-1})$ the set $O(\alpha)=\cap_j T^{-j}(I_j)$, for all $x\in[0,1]$ there is $\alpha\in G^n$ such that $x\in O(\alpha)$, and by construction each $O(\alpha)$ has diameter at most $\varepsilon$ for the dynamical metric $d_n$. 
As is well-known, $N(d_h,T,\varepsilon,n)$ is bounded above by $\card G^n$: no two elements of any given $(\varepsilon,n)$-separated set can lay on the same element of any given $(\varepsilon,n)$-cover.

For $\al=(I_t)_{0\le t<n} \in G^n$, we consider its ``depth/label sequence''
\[L(\al):=(\ell(I_t))_{0\le t<n} \,,\]
and for each $L=(\ell_t)_{0\le t< n}\in\llbracket 1, K_\varepsilon \rrbracket^{n}$ we denote by $A_L$ the set of paths $\alpha\in G^n$ such that $L(\alpha)=L$. Thanks to Lemma \ref{lemm:depth}, it is sufficient to estimate the cardinal of the $A_L$ since
\begin{align*}
\log_2 \card G^n &= \log_2 \sum_{L\in \llbracket 1, K_\varepsilon \rrbracket^{n}} \card A_L \\
 &\le n\log_2(K_\varepsilon) + \max \big\{\log_2(\card A_L) \colon L\in \llbracket 1, K_\varepsilon \rrbracket^{n} \big\} \,,\\
\frac{\log_2 \card G^n}{n\log_2 \frac1\varepsilon} &\le \underbrace{\frac{\log_2(\log_2\frac1\varepsilon+3)}{\log_2 \frac1\varepsilon}}_{\to 0}+ \frac{\max \big\{\log_2(\card A_L) \colon L\in \llbracket 1, K_\varepsilon \rrbracket^{n} \big\}}{n\log_2 \frac1\varepsilon} \,.
\end{align*}

Fix any $L=(\ell_0,\dots,\ell_{n-1}) \in \llbracket 1,K_\ve \rrbracket^{n}$. The cardinality of $A_L$ is easy to bound in some specific cases:
\begin{enumerate}
\item\label{enumi:small} If the depth stays below $p_0$, i.e. $L \in \llbracket 1,p_0 \rrbracket^{n}$: then by Lemma \ref{lemm:cardL} and inequality \eqref{equ:n2},
\[
\frac{\log_2 \card A_L}{n} \le \log_2 4^{p_0} \le 2 \sqrt{\log_2 \frac1\ve} \,.
\]
\item If the depth starts above $p_0$ but does not raises above this starting value and the run is short enough, i.e. if $L\in \{k\}\times \llbracket 1,k \rrbracket^{n-1}$ for some $k \in\llbracket p_0,K_\varepsilon\rrbracket$ and $n\le k$: then By Lemma \ref{lemm:cardL} and Corollary \ref{coro:orbitBound}, 
\[
\frac{\log_2 \card A_L}{n} \le 2 \frac{k+3n}{n} \,,
\] 
which will be small if $n$ is equal (or close to) $k$.
\item\label{enumi:block} More generally, if the depth increases several times above the previous extremal value within the time window given by Corollary \ref{coro:orbitBound}, i.e. if $L\in \prod_{i=1}^q \{k_i\}\times \llbracket 1,k_i \rrbracket^{n_i-1}$ for some $p_0<k_1<\dots<k_q\le K_\varepsilon$, $(n_i)_{1\le i\le q}$ with $n_i\le k_i$ for all $i$ and $n=\sum_i n_i\ge k_q$, then with the convention $k_0=0$:
\begin{align*}
\frac{\log_2 \card A_L}{n} 
  &\le \frac{1}{n} \log_2\Big(\prod_{i=1}^{q} 4^{k_i-k_{i-1}+3}\cdot 4^{3 (n_i-1)}  \Big)\\
  &\le \frac2n \sum_i (k_i-k_{i-1}+3n_i) \\
  &\le \frac{2 k_q}{k_q}+6 = 8 < 2\sqrt{\log_2\frac1\varepsilon} \,,
\end{align*}
\end{enumerate}
as soon as $\varepsilon$ is small enough.
We now cut $L$ into blocks for which case \ref{enumi:small} or \ref{enumi:block} applies (except possibly for the last one, which can be too short for case \ref{enumi:block}). Define $t_0=0$ and recursively:
\begin{itemize}
\item if $\ell_{t_i}\le p_0$, $t_{i+1}$ is the first time after $t_i$ at which $\ell_{t_{i+1}}>p_0$ (so that $L_i:=(\ell_j)_{t_i\le j<t_i+1}$ falls into case \ref{enumi:small} above),
\item if $\ell_{t_i}>p_0$, $t_{i+1}$ is the largest integer after $t_i$ such that $L_i$ belongs to case \ref{enumi:block} above,
\end{itemize}   
until no such time exists, at which point we set $t_{i+1}=n$ and $i+1=:m$.
We obtain a sequence of integers $0=t_0 < t_1 < t_2 < \dots < t_m=n$ such that for each $i\in\llbracket 0,m-2\rrbracket$, the subsequence $L_i$ satisfies
\[\frac{\log_2\card A_{L_i}}{(t_{i+1}-t_i)} \le 2\sqrt{\log_2\frac1\varepsilon} \,.\]
By taking $n>K_\varepsilon^4$, we ensure that $t_m-t_{m-1}\le \sqrt{n}$ so that 
\[\frac{\log_2 \card A_{L_{m-1}}}{n} \le \frac{\log K_\varepsilon}{\sqrt{n}} \,.\]
Then finally, recalling the definition $K_\ve= \lceil \log_2 \frac1\ve \rceil+2$, we have
\begin{align*}
\frac{\log \card A_L}{n} &= \frac1n \sum_{i=1}^{m-1} \log \card A_{L_i} \\
  &\le \frac{\log \card A_{L_m}}{n} + \frac{1}{t_{m-1}} \sum_{i=0}^{m-2} \log \card A_{L(t_i)} \\
  &\le \frac{\log  K_\varepsilon}{\sqrt{n}} +
   \max_{0\le i\le m-2}\Big\{ \frac{\log \card A_{L_i}}{t_{i+1}-t_{i}}  \Big\} \,,\\
\limsup_{n \to \infty} \frac{\log \card A_L}{n\log_2\frac1\varepsilon}  &\le \frac{2}{\sqrt{\log_2\frac1\varepsilon}} \to 0 \quad(\text{as }\varepsilon\to0).
\end{align*}

\bibliographystyle{alpha}
\bibliography{rough}

\end{document}